\documentclass[12pt]{amsart}
\usepackage{a4wide,graphicx}
\usepackage{color}
\usepackage{caption}
\usepackage[labelformat=simple,labelfont={}]{subfig}  

\let\pa\partial  
\let\na\nabla  
\let\eps\varepsilon  
\newcommand{\N}{{\mathbb N}}  
\newcommand{\R}{{\mathbb R}} 
\newcommand{\Z}{{\mathbb Z}} 
\newcommand{\diver}{\operatorname{div}}  
\newcommand{\dx}{\mathrm{d}x}
\newcommand{\T}{{{\mathbb T}^d}} 
\newcommand{\K}{{\mathbb T}}
\newcommand{\dd}{{\mathrm{d}}}

\newtheorem{theorem}{Theorem}   
\newtheorem{lemma}[theorem]{Lemma}   
   
\newtheorem{remark}[theorem]{Remark}   
\newtheorem{corollary}[theorem]{Corollary}

\begin{document}  

\title[Approximations of a quantum diffusion equation]{Entropy-stable 
and entropy-dissipative approximations of a fourth-order quantum diffusion equation}

\author{Mario Bukal}
\address{Institute for Analysis and Scientific Computing, Vienna University of  
	Technology, Wiedner Hauptstra\ss e 8--10, 1040 Wien, Austria}
\email{mbukal@asc.tuwien.ac.at} 

\author{Etienne Emmrich}
\address{Institute for Mathematics, Technical University of Berlin,
  Stra{\ss}e des 17. Juni 136, 10623 Berlin, Germany}
\email{emmrich@math.tu-berlin.de}

\author{Ansgar J\"ungel}
\address{Institute for Analysis and Scientific Computing, Vienna University of  
	Technology, Wiedner Hauptstra\ss e 8--10, 1040 Wien, Austria}
\email{juengel@tuwien.ac.at} 

\date{\today}

\thanks{The first and last author acknowledge partial support from   
the Austrian Science Fund (FWF), grants P20214, P22108, and I395, and    
the Austrian-French Project of the Austrian Exchange Service (\"OAD)} 

\begin{abstract}
Structure-preserving numerical schemes for a nonlinear parabolic fourth-order
equation, modeling the electron transport in quantum semiconductors, with
periodic boundary conditions are analyzed. First, a two-step  
backward differentiation formula (BDF) semi-discretization
in time is investigated. The scheme preserves the nonnegativity of the solution,
is entropy stable and dissipates a modified
entropy functional. The existence of a weak semi-discrete solution and, in
a particular case, its temporal second-order convergence to the continuous
solution is proved. The proofs employ an algebraic relation which implies the G-stability of the two-step BDF.
Second, an implicit Euler and $q$-step BDF discrete
variational derivative method are considered. This scheme, which exploits the
variational structure of the equation, dissipates 
the discrete Fisher information (or energy). Numerical experiments show that
the discrete (relative) entropies and Fisher information decay even exponentially 
fast to zero.
\end{abstract}

\keywords{Derrida-Lebowitz-Speer-Spohn equation, discrete entropy-dissipation
i\-ne\-qua\-li\-ty, Fisher information, BDF time discretization, numerical convergence, 
discrete variational derivative method.}  
 
\subjclass[2000]{65M06, 65M12, 65M15, 35Q40, 82D37.}  

\maketitle


\section{Introduction}

This paper is devoted to the study of novel structure-preserving 
temporal higher-order numerical schemes for the fourth-order quantum diffusion equation
\begin{equation}\label{qde}
  n_t + \diver\left(n\na\left(\frac{\Delta\sqrt{n}}{\sqrt{n}}\right)\right) = 0,
  \quad x\in\T,\ t>0, \quad n(0)=n_0,
\end{equation}
where $\T$ is the $d$-dimensional torus.
This equation is the zero-temperature and zero-field limit of the quantum
drift-diffusion model, which describes the evolution of the electron density 
$n(t)=n(t,\cdot)$ in a quantum semiconductor device; see \cite{Jue09}. 
It was derived in \cite{DMR05} from a relaxation-time Wigner equation 
using a Chapman-Enskog expansion around the quantum equilibrium.
For smooth positive solutions, \eqref{qde} can be written in a symmetric 
form for the variable $\log n$:
\begin{equation}\label{dlss}
  n_t + \frac12\pa_{ij}^2(n\pa_{ij}^2\log n) = 0,\quad x\in\T,\ t>0,\quad
  n(0)=n_0,
\end{equation}
where here and in the following, we employ the summation convention over 
repeated indices and the notation $\pa_i=\pa/\pa x_i$, 
$\pa_{ij}^2=\pa^2/\pa x_i\pa x_j$. This is the multidimensional form of the so-called
Derrida-Lebowitz-Speer-Spohn (DLSS) equation.
Its one-dimensional version was derived in \cite{DLSS91} in a suitable
scaling limit from the time-discrete Toom model and the variable $n$ is related
to a limit random variable.

The main difficulties in the analysis of \eqref{qde} (or \eqref{dlss}) 
are the highly nonlinear
structure, originating from the quantum potential term $\Delta\sqrt{n}/\sqrt{n}$
in \eqref{qde}, and the fourth-order differential operator, which lacks
a maximum principle.

These difficulties have been overcome by exploiting the rich 
mathematical structure of \eqref{dlss}.
First, equation \eqref{dlss} preserves the nonnegativity
of the solutions \cite{JuMa08}: Starting from a nonnegative initial datum, the
weak solution stays nonnegative for all time. Second, \eqref{dlss} allows for a class of Lyapunov 
functionals and so-called entropy dissipation estimates. 
More precisely, the functionals
$$
  E_\alpha[n] = \frac{1}{\alpha(\alpha-1)}\int_\T n^\alpha \dx \quad
  (\alpha\neq 0,1),\quad E_1[n] = \int_\T\big(n(\log n-1)+1\big)\dx
$$
are Lyapunov functionals along solutions to \eqref{dlss}, i.e.\ 
$\dd E_\alpha[n]/\dd t\le 0$
if $(\sqrt{d}-1)^2/(d+2)\le\alpha\le (\sqrt{d}+1)^2/(d+2)$, and the entropy
dissipation inequality
\begin{equation}\label{1.est1}
  \frac{\dd}{\dd t}E_\alpha[n] + 2\kappa_\alpha
  \int_\T(\Delta n^{\alpha/2})^2\dx \le 0
\end{equation}
holds if $(\sqrt{d}-1)^2/(d+2)<\alpha < (\sqrt{d}+1)^2/(d+2)$.
The constant $\kappa_\alpha>0$ can be computed explicity, see Lemma \ref{lem.H2} below.
For $\alpha=1$, inequality \eqref{1.est1}
can be interpreted as the dissipation of the physical entropy.
Third, equation \eqref{qde} is the gradient flow of the Fisher information
\begin{equation}\label{1.F}
  F[n] = \int_\T|\na\sqrt{n}|^2\dx
\end{equation}
with respect to the Wasserstein metric \cite{GST09}. As the variational
derivative of the Fisher information equals 
$\delta F[n]/\delta n=-\Delta\sqrt{n}/\sqrt{n}$, a straightforward computation
shows that the Fisher information is dissipated along solutions to \eqref{qde},
\begin{equation}\label{1.est2}
  \frac{\dd}{\dd t}F[n] + \int_\T n\Big|\na\Big(\frac{\delta F[n]}{\delta n}
  \Big)\Big|^2\dx = 0.
\end{equation}
Since the Fisher information can be interpreted as the quantum energy,
the latter can be seen as an energy dissipation identity.

Whereas the local-in-time existence of positive classical solutions
for strictly positive $W^{1,p}(\T)$ initial data with $p>d$ could be proved
using semigroup theory \cite{BLS94}, global-in-time existence results
were based on estimates \eqref{1.est1} and \eqref{1.est2}.
More precisely, the global existence of a nonnegative
weak solution was achieved in \cite{JuPi00} in the one-dimensional case. This
result was extended later to several space dimensions in
\cite{JuMa08}, employing entropy dissipation inequalites, and in \cite{GST09},
exploring the variational structure of the equation.

From a numerical viewpoint, it is desirable to design numerical
approximations which preserve the above structural properties like positivity
preservation, entropy stability, and entropy or energy dissipation on a discrete level. For a constant time step size $\tau >0$, let $t_k = k\tau$ $(k\ge 0)$.
If $n_k$ approximates the solution $n(t_k)$ to \eqref{dlss} at time $t_k$,
we call a numerical scheme {\em entropy dissipating}
if $E_\alpha[n_{k+1}]\le E_\alpha[n_k]$ for all $k\ge 0$ with $\alpha$ in a certain
parameter range, and {\em entropy stable} if there exists a constant
$C>0$ such that $E_\alpha[n_k]\le C$ for all $k\ge 0$. In this paper, we 
investigate the entropy stability and entropy 
dissipation of backward differentiation formulas (BDF).

In the literature, most of the numerical schemes proposed for \eqref{dlss}
are based on an implicit Euler discretization in one space dimension. 
In \cite{JuPi01}, the convergence of a positivity-preserving semi-discrete 
Euler scheme was shown.
A fully discrete finite-difference scheme which preserves the positivity, 
mass, and physical entropy was derived in \cite{CJT03}. 
D\"uring et al.\ \cite{DMM10} employed the variational structure of \eqref{dlss}
on a fully discrete level and introduced a discrete minimizing movement scheme.
This approach implies the decay of the discrete Fisher information and the
nonnegativity of the discrete solutions. 
Finally, a positivity-preserving finite-volume scheme in several space dimensions 
for a stationary quantum drift-diffusion model was suggested in \cite{CGJ11}.

Positivity preserving and entropy consistent numerical schemes have been
investigated in the literature also for other nonlinear fourth- and second-order
equations. For instance, a positivity preserving finite difference approximation
of the thin-film equation was proposed by Zhornitskaya and Bertozzi \cite{ZhBe00}.
Finite element techniques for the same equation were employed by Barrett,
Blowley, and Garcke \cite{BBG98}, imposing the nonnegativity property as a constraint such that at each time level
a variational inequality has to be solved.
Furthermore, entropy consistent finite volume--finite element schemes were
suggested and analyzed by Gr\"un and Rumpf \cite{Gru03,GrRu00}.
Furihata and Matsuo \cite{FuMa10} developed the
discrete variational derivative method to derive conservative or dissipative
schemes for a variety of evolution equations possessing a variational structure.
Entropy dissipative fully discrete schemes for electro-reaction-diffusion
systems were derived by Glitzky and G\"artner \cite{GlGa09}.

In most of these works, the time discretization is restricted to the implicit
Euler method, motivated by the fact that the solutions often lack regularity. 
However, high-order schemes often still yield smaller time errors than the Euler
scheme, and this improved accuracy is vital to match the spatial approximation
errors. A difficulty of the analysis is that the time discretization has to be
compatible with the entropy structure of the equation. This is the case
for the first-order implicit Euler discretization. Indeed, multiplying the semi-discrete scheme
\begin{equation}\label{1.euler}
  \frac{1}{\tau}(n_{k+1}-n_k) + \frac12\pa_{ij}^2(n_{k+1}\pa_{ij}^2\log n_{k+1})
  = 0, \quad k\ge 0,
\end{equation}
where $\tau>0$ is the time step and $n_k$ approximates $n(t_k)$ with $t_k=\tau k$,
by $\log n_{k+1}$ and using the elementary inequality 
\begin{equation}\label{euler.ineq}
  (x-y)\log x\ge x\log x-y\log y\quad \mbox{for }x,y>0
\end{equation} 
(which follows from the convexity of $x\mapsto x\log x$), 
it was shown in \cite[Lemma 4.1]{JuMa08} that
$$
  E_\alpha[n_{k+1}] + 2\tau\kappa_\alpha\int_\T (\Delta n_{k+1}^{\alpha/2})^2\dx
  \le E_\alpha[n_k], \quad k\ge 0.
$$
As a consequence, $k\mapsto E_\alpha[n_k]$ is nonincreasing and the entropy
dissipation structure is preserved. It is less clear whether higher-order approximations
yield entropy dissipating numerical schemes. In this paper, we prove this
property for the two-step BDF method.

Two-step BDF (or BDF2) methods have been employed in the literature to
approximate various evolution equations in different contexts. 
We just mention numerical schemes for
incompressible Navier-Stokes problems \cite{Emm04,GiRa81,HiSu00},
semilinear and quasilinear parabolic equations \cite{Emm05,Moo94}, and
nonlinear evolution problems governed by monotone operators \cite{Emm09,Kre78}.
To our knowledge, temporal higher-order schemes for the quantum diffusion 
equation \eqref{qde} have been not considered so far.

In the following, we detail our main results. 
First, we analyze the BDF2 time approximation of the DLSS equation, 
written in the form
\begin{equation}\label{alpha_dlss}
  \frac{2}{\alpha}n^{1-\alpha/2}(n^{\alpha/2})_t 
  + \frac12\pa_{ij}^2(n\pa_{ij}^2\log n) = 0,
\end{equation}which was already used in \cite{JuVi07} in a different context. 
Introducing the variable $v_k:=n_k^{\alpha/2}$, which approximates $n(t_k)^{\alpha/2}$, 
the semi-discrete BDF2 scheme for \eqref{alpha_dlss} reads as
\begin{equation}\label{disc.dlss}
  \frac{2}{\alpha\tau}v_{k+1}^{2/\alpha-1}\left(\frac32 v_{k+1} - 2v_k 
  + \frac12 v_{k-1}\right) + \frac{1}{2}\pa_{ij}^2(n_{k+1}\pa_{ij}^2
  \log n_{k+1}) = 0 \quad\mbox{in }\T,\ k\geq 1.
\end{equation}
Here, $v_0=n_0^{\alpha/2}$ is given by the initial datum $n_0$, and $v_1$
is the solution to the implicit Euler scheme
\begin{equation}\label{alpha_euler}
  \frac{2}{\alpha\tau}v_1^{2/\alpha-1}\big(v_{1} - v_0\big) 
  + \frac12\pa_{ij}^2(n_{1}\pa_{ij}^2\log n_{1}) = 0
  \quad\mbox{in }\T.
\end{equation}
The existence of a weak solution to the scheme \eqref{disc.dlss}--\eqref{alpha_euler} 
is provided by the following theorem.

\begin{theorem}[Existence of solutions and entropy stability]\label{thm.bdf2.ex}
Let $1\le d\le 3$, $1 \leq \alpha < (\sqrt{d}+1)^2/(d+2)$, and let $n_0 \in L^3(\T)$ 
be a nonnegative function.
Then there exists a weak solution $v_1 = n_1^{\alpha/2}$ of the implicit Euler scheme
\eqref{alpha_euler} and a sequence $(v_k)=(n_k^{\alpha/2})$ 
of weak nonnegative solutions to \eqref{disc.dlss} satisfying
$v_k \geq 0$ in $\T$, $v_k\in H^2(\T)$,
and for all $\phi\in W^{2,\infty}(\T)$,
\begin{align}\label{weak_alpha}
  \frac{1}{\alpha\tau}\int_{\T} & v_{k+1}^{2/\alpha-1}\left(\frac32 v_{k+1} 
  - 2v_k + \frac12 v_{k-1}\right)\phi\dx \\ 
  &{}+ \int_{\T}\left(\frac{1}{2\alpha}v_{k+1}^{2/\alpha-1}\pa_{ij}^2v_{k+1} 
  - \frac{\alpha}{2}\pa_i (v_{k+1}^{1/\alpha})\pa_j
  (v_{k+1}^{1/\alpha})\right)\pa_{ij}^2\phi\dx = 0. \nonumber
\end{align}
If $\alpha > 1$, the scheme \eqref{disc.dlss} is entropy stable and
the a priori estimate 
\begin{equation}\label{ent.stab}
  E_\alpha[n_m] + \frac43\kappa_\alpha\tau\sum_{k=1}^{m}\int_{\T}
  \big(\Delta(n_{k}^{\alpha/2})\big)^2\dx  
  \le E_\alpha[n_0],\quad m\geq 1,
\end{equation}
holds, where $\kappa_\alpha>0$ is defined in Lemma \ref{lem.H2}.
\end{theorem}

When we redefine the entropy, we are able to prove entropy dissipation
of the semi-discrete scheme. For this, introduce the modified entropy
$$
  E^G_\alpha[n_k,n_{k-1}] = \frac{1}{2\alpha(\alpha-1)}\int_\T
  \big(n_k^\alpha + (2n_k^{\alpha/2}-n_{k-1}^{\alpha/2})^2\big)\dx, \quad k\ge 1.
$$
This definition is motivated by the inequality 
$$
  2\left(\frac32 a-2b+\frac12 c\right)a \ge \frac12\big(a^2+(2a-b)^2\big)
  - \frac12\big(b^2+(2b-c)^2\big) \quad\mbox{for all }a,b,c\in\R,
$$
which implies the G-stability of the BDF2 method; see \cite{Dah78} and 
Lemma \ref{lem.ineq}.
The entropies $E_\alpha$ and $E_\alpha^G$ are formally related by 
$E^G_\alpha[n_k,n_{k-1}] = E_\alpha[n_k] + O(\tau)$ as $\tau\to 0$ for $k\ge 2$.

\begin{corollary}[Entropy dissipation]\label{coro.bdf2}
Let the assumptions of Theorem \ref{thm.bdf2.ex} hold for $\alpha>1$.
Then the scheme \eqref{disc.dlss} is entropy dissipative in the sense of
\begin{equation}\label{ent.diss}
  E^G_\alpha[n_{k+1},n_k] + 2\kappa_\alpha\tau\int_{\T}
  \big(\Delta(n_{k+1}^{\alpha/2})\big)^2\dx  
  \le E^G_\alpha[n_k,n_{k-1}],\quad k\geq 1.
\end{equation}
In particular, $k\mapsto E^G_\alpha[n_k,n_{k-1}]$ is nonincreasing.
\end{corollary}

We stress the fact that the implicit Euler scheme \eqref{1.euler} dissipates
{\em all} admissible entropies, whereas the BDF2 scheme just dissipates {\em one}
entropy, $E^G_\alpha[n_k]$, where $\alpha$ has been fixed in the scheme. 

The proof of Theorem \ref{thm.bdf2.ex} is based on the semi-discrete 
entropy stability inequality \eqref{ent.stab}
and the Leray-Schauder fixed-point theorem. Instead of
\eqref{euler.ineq}, we employ the algebraic inequalities \eqref{ineq1} and
\eqref{ineq2} (see Section \ref{sec.bdf2}). We have not been able to obtain
similar inequalities for BDF$k$ methods with $3\le k\le 6$. The reason
might be the fact that the only G-stable BDF methods are the BDF1 (implicit Euler)
and BDF2 discretizations \cite{Dah78}. Moreover, we have not been able
to prove entropy dissipation for $\alpha=1$ since in this case, 
inequalities \eqref{ineq1} and \eqref{ineq2} cannot be used.

If $\alpha=1$, we prove that the semi-discrete solution to the BDF2 scheme 
converges to the continuous solution with second-order rate.

\begin{theorem}[Second-order convergence]\label{thm.bdf2.conv}
Let the assumptions of Theorem \ref{thm.bdf2.ex} hold, let $\alpha=1$, and
let $(v_k)$ be the sequence of solutions to \eqref{disc.dlss}-\eqref{alpha_euler} 
constructed in Theorem \ref{thm.bdf2.ex}. 
We assume that there exist values $\mu_k>0$ such that
$v_k\ge \mu_k>0$ in $\T$. Furthermore, let $n$ be a strictly positive solution
to \eqref{dlss} satisfying $\sqrt{n}\in H^3(0,T;L^2(\T))\cap W^{2,\infty}(0,T;
L^2(\T))$. Then there exists a constant $C>0$, 
depending only on the $L^2(0,T;L^2(\T))$
norm of $(\sqrt{n})_{ttt}$, the $L^\infty(0,T;L^2(\T))$ norm of $(\sqrt{n})_{tt}$,
and $T$, but not on $\tau$, such that
$$
  \|v_k - \sqrt{n(t_k,\cdot)}\|_{L^2(\T)} \le C\tau^2,
$$
where $0<\tau<1/8$ is the time step and $t_k=\tau k$, $k\ge 0$.
\end{theorem}

It is shown in \cite[Theorem 6.2]{BLS94} that the solution $n$ to \eqref{dlss}
is smooth locally in time if the initial datum is positive and an element of
$W^{1,\infty}(\T)$. The proof of Theorem \ref{thm.bdf2.conv} is based on 
local truncation error estimates and the monotonicity of the formal operator
$A(v)=v^{1-2/\alpha}\pa_{ij}^2(v^2\pa_{ij}^2\log v)$ for $\alpha=1$ \cite{JuPi03}.
If $\alpha\neq 1$, the operator $A$ seems to be not monotone, and our proof does
not apply. Possibly, the second-order convergence for $\alpha\neq 1$ could be achieved by applying suitable nonlinear semigroup estimates.

Next, we investigate a fully discrete numerical scheme which dissipates the
Fisher information. To this end, we employ the discrete variational derivative
method of Furihata and Matsuo \cite{FuMa10}. The method
is based on the variational structure of the DLSS equation,
\begin{equation}\label{dlss.gf}
  n_t + \diver\left(n\na\frac{\delta F[n]}{\delta n}\right) = 0, \quad t>0.
\end{equation}
The dissipation of the Fisher information $F[n]$ (see \eqref{1.F})
follows from (formally) integrating by parts in
$$
  \frac{\dd}{\dd t}F[n]
  = \int_\T\frac{\delta F[n]}{\delta n}n_t \dx
  = -\int_\T n\left|\na\left(\frac{\delta F[n]}{\delta n}\right)\right|^2\dx \le
  0 $$
(see \eqref{1.est2}). The idea of the method is to derive a discrete formula
for the variational derivative $\delta F[n]/\delta n$ in such a way that the
above integration by parts formula and consequently the dissipation
property hold on a discrete level. We provide such formulas for spatial
finite difference and temporally higher-order BDF approximations.


The numerical approximation for equation (\ref{qde}), derived in
\cite{DMM10}, takes advantage of the gradient-flow structure in the sense
that the variational structure was discretized instead of equation (\ref{qde})
itself. The method is based on the minimizing movement (steepest descent) scheme
and consequently dissipates the discrete Fisher information. 
In each time step, a constrained quadratic optimization problem for the Fisher 
information needs to be solved on a finite-dimensional space. Each subproblem has 
to be solved iteratively, leading to a sequential quadratic programming method. 
In general, this structure-preserving approach, known as ``first discretize, 
then minimize'', has good stability properties and captures well other structural
features of equations, like those presented in \cite{WW10}.

The strategy of the discrete variational derivative method is the standard
``first minimize, then discretize'' approach, i.e., the discretization of
equation (\ref{qde}), as the minimality condition in the variational setting, is
performed. To some extent this is simpler than the above approach, 
since in each time step only a discrete 
nonlinear system has to be solved, and the main structural property remains
preserved. Furthermore, we derive temporally higher-order discretizations,
whereas the scheme in \cite{DMM10} is of first order only.

To simplify the notation, we consider the spatially one-dimensional case only.
The extension to the multidimensional situation is straightforward if we assume
rectangular grids.
Let $x_0,\ldots,$ $x_N$ be equidistant grid points of $\K$ with mesh size
$h>0$ and $x_0\cong x_N$. Let $U_i^k$ be an
approximation of $n(t_k,x_i)$ and set $U^k=(U^k_0,\ldots,U^k_{N-1})$,
$U_N=U_0$. Furthermore, let $\delta_k^{1,q}$ be the
$q$-step BDF operator at time $t_k$; for instance,
\begin{align}
  \delta_{k+1}^{1,q}U_i^{k+1} &= \frac{1}{\tau}(U^{k+1}_i-U^{k}_i) 
  \quad\mbox{if }q=1, \label{1.bdf1} \\
  \delta^{1,q}_{k+1}U_i^{k+1} 
  &= \frac{1}{\tau}\left(\frac32 U^{k+1}_i - 2U^k_i + \frac12 U^{k-1}_i\right)
  \quad\mbox{if }q=2. \label{1.bdf2}
\end{align}
We denote by $\delta^{\langle 1\rangle}_i$ the central
finite-difference operator at $x_i$, i.e.\ $\delta^{\langle 1\rangle}_i U^k
= (U_{i+1}^k-U_{i-1}^k)/h$.
Then, following \eqref{dlss.gf}, we propose the fully discrete scheme
\begin{equation}\label{dvdm.q}
  \delta_{k+1}^{1,q}U_i^{k+1} = \delta^{\langle 1\rangle}_i
  \left(U^{k+1}\delta^{\langle 1\rangle}_i\left(
  \frac{\delta F_d}{\delta(U^{k+1},\ldots,U^{k-q+1})}\right)\right), \quad
  k\ge q-1,
\end{equation}
where $i=0,\ldots,N-1$. The discrete variational derivative
$\delta F_d/\delta(U^{k+1},\ldots,U^{k-q+1})\in\R^N$ is defined in such a way
that a discrete chain rule holds (see \eqref{q1.dvd} and \eqref{gen.dvd}
in Section \ref{sec.dvdm} for the precise definitions), yielding the dissipation of
the discrete Fisher information $F_d[U^k]$ in the sense of the following theorem.

\begin{theorem}[Dissipation of the Fisher information]\label{thm.dvdm}
Let $N\in\N$, $U^0\in\R^N$ be some nonnegative initial datum with unit mass, 
$\sum_{i=0}^{N-1} U_i^0 h=1$,
and let $U^1,\ldots,U^{q-1}\in\R^N$ be starting values with unit mass and
$F_d[U^{q-1}]\le\cdots\le F_d[U^0]<\infty$. Then scheme \eqref{dvdm.q},
with the discrete variational derivative
$\delta F_d/\delta(U^{k+1},\ldots,U^{k-q+1})$ defined by \eqref{gen.dvd},
is consistent of order $(q,2)$ with respect to the time-space discretization.
Furthermore, $U^k$ is bounded uniformly in $k$, has unit mass, and
the discrete Fisher information is dissipated in the sense of
$$
  \delta^{1,q}_k F_d[U^k] \le 0 \quad\mbox{for all }k\ge q.
$$
Furthermore, for $q=1$ the discrete variational derivative is defined by
\eqref{q1.dvd}, scheme \eqref{dvdm.q} is consistent of order $(1,2)$ and
the discrete Fisher information is nonincreasing, $F_d[U^{k+1}]\le F_d[U^k]$ 
for all $k\ge 1$.
\end{theorem}

We say that a scheme is consistent of order $(q,m)$ if the truncation error
is of the order $O(\tau^q)+O(h^m)$ for $\tau\to 0$ and $h\to 0$.

The paper is organized as follows. The analysis of the BDF2 time approximation
is performed in Section \ref{sec.bdf2}, and Theorems \ref{thm.bdf2.ex} and
\ref{thm.bdf2.conv} are proved. The fully discrete variational derivative method
is detailed in Section \ref{sec.dvdm}, and Theorem \ref{thm.dvdm} is proved.
Numerical experiments in Section \ref{sec.num} illustrate the entropy stability,
entropy dissipation, and energy (Fisher information) dissipation, even in
situations not covered by the above theorems.


\section{BDF2 time approximation}\label{sec.bdf2}

First, we collect some auxiliary results. The following lemma is needed
to show a priori bounds for the semi-discrete solutions to the DLSS equation.

\begin{lemma}\label{lem.ineq}
It holds for all $a$, $b$, $c\in\R$,
\begin{align}
  2\left(\frac32 a-2b+\frac12 c\right)a &\ge \frac32 a^2-2b^2+\frac12 c^2
  + (a-b)^2 - (b-c)^2, \label{ineq1} \\
  2\left(\frac32 a-2b+\frac12 c\right)a &\ge \frac12\big(a^2+(2a-b)^2\big)
  - \frac12\big(b^2+(2b-c)^2\big). \label{ineq2}
\end{align}
\end{lemma}

\begin{proof}
We calculate
$$
  2\left(\frac32 a-2b+\frac12 c\right)a = \frac32 a^2-2b^2+\frac12 c^2
  + (a-b)^2 - (b-c)^2 + \frac12(a-2b+c)^2,
$$
which proves the first assertion. Because of
$$
  2\left(\frac32 a-2b+\frac12 c\right)a = \frac12\big(a^2+(2a-b)^2\big)
  - \frac12\big(b^2+(2b-c)^2\big) + \frac12(a-2b+c)^2,
$$
the second assertion follows as well.
\end{proof}

We also recall the following inequality (see \cite[Lemma 2.2]{JuMa08} for a proof).

\begin{lemma}\label{lem.H2}
Let $d\ge 2$ and $\sqrt{n}\in H^2(\T)\cap W^{1,4}(\T)\cap L^\infty(\T)$ with
$\inf_\T n>0$. Then, for any $(\sqrt{d}-1)^2/(d+2)<\alpha<(\sqrt{d}+1)^2/(d+2)$,
$\alpha\neq 1$, $$
  \frac{1}{4(\alpha-1)}\int_\T n\pa_{ij}^2(\log n)\pa_{ij}^2(n^{\alpha-1})\dx
  \ge \kappa_\alpha\int_\T(\Delta n^{\alpha/2})^2 \dx
$$
and for $\alpha=1$,
$$
  \frac14\int_\T n(\pa_{ij}^2(\log n))^2 \dx \ge \kappa_1\int_\T(\Delta\sqrt{n})^2\dx,
$$
where
$$
  \kappa_\alpha = \frac{p(\alpha)}{\alpha^2(p(\alpha)-p(0))}>0 \quad\mbox{and}\quad
  p(\alpha) = -\alpha^2 + \frac{2(d+1)}{d+2}\alpha - \left(\frac{d-1}{d+2}\right)^2.
$$
\end{lemma}

\begin{proof}[Proof of Theorem \ref{thm.bdf2.ex}.]
Given $v_0=n_0^{\alpha/2}$, the existence of a nonnegative weak solution
$v_1\in H^2(\T)$ to \eqref{alpha_euler} is shown in \cite{JuMa08}. 
Assume that $v_2,\ldots,v_{k}\in H^2(\T)$ are solutions to \eqref{weak_alpha}.
We introduce the variable $y$ by $v_{k+1}=e^{\alpha y/2}$ such that $n_{k+1}=e^y$. First, we prove 
the existence of a weak solution $y\in H^2(\T)$ to the regularized equation
\begin{equation}\label{reg.form}
  \frac{2}{\alpha\tau}e^{(1-\alpha/2)y}\left(\frac32e^{\alpha y/2} - 2v_k 
  + \frac12 v_{k-1}\right) + \frac12\pa_{ij}^2(e^y\pa_{ij}^2y) 
  + \eps L(y) = 0,
\end{equation}
where $\eps>0$ and 
$$
L(y) = \Delta^2 y 
  - \diver(|\nabla y|^2\nabla y) 
  + y. 
$$

{\em Step 1: Definition of the fixed-point operator.} Given $z\in W^{1,4}(\T)$ and
$\sigma\in [0,1]$, we define on $H^2(\T)$ the forms
\begin{align*}
  a(y,\phi) &= \frac{1}{2}\int_{\T}e^z\pa_{ij}^2y\pa_{ij}^2\phi \dx 
  + \eps\int_{\T}\big(\Delta y\Delta\phi 
  + |\nabla z|^2\na y\cdot\na\phi + y\phi\big)\dx, \\
  f(\phi) &= -\frac{2\sigma}{\alpha\tau}\int_{\T}e^{(1-\alpha/2)z}
  \left(\frac32 e^{\alpha z/2} - 2v_{k} + \frac12 v_{k-1}\right)\phi \dx.
\end{align*}
Since $H^2(\T)\hookrightarrow W^{1,4}(\T)\hookrightarrow L^\infty(\T)$ with
continuous embeddings (remember that $d\le3$), these mappings are well defined and continuous. Furthermore,
by the Poincar\'e inequality for periodic functions with constant $C_P>0$,
the bilinear form $a$ is coercive,
\begin{align*}
  \eps\|y\|_{H^2(\T)}^2 &= \eps\int_\T\big(|\na^2 y|^2 + |\na y|^2 + y^2\big)\dx
  \le \eps\int_\T\big((C_P^2+1)|\na^2 y|^2 + y^2\big)\dx \\
  &= \eps\int_\T\big((C_P^2+1)(\Delta y)^2 + y^2\big)\dx
  \le \eps(C_P^2+1)\int_\T((\Delta y)^2 + y^2)\dx \\
  &\le (C_P^2+1) a(y,y).
\end{align*}
By Lax-Milgram's lemma, there exists a unique solution $y\in H^2(\T)$ to
$$
  a(y,\phi) = f(\phi) \quad\mbox{for all }\phi\in H^2(\T).
$$
This defines the fixed-point operator $S:W^{1,4}(\T)\times[0,1]\to W^{1,4}(\T)$,
$S(z,\sigma)=y$. It holds $S(y,0)=0$ for all $y\in W^{1,4}(\T)$, and $S$
is continuous and compact, in view of the compact embedding
$H^2(\T)\hookrightarrow W^{1,4}(\T)$. In order to apply the Leray-Schauder theorem, 
it remains to show that there exists
a uniform bound in $W^{1,4}(\T)$ for all fixed points of $S(\cdot,\sigma)$.

{\em Step 2: A priori bound.} Let $y\in H^2(\T)$ be a 
fixed point of $S(\cdot,\sigma)$ for some $\sigma\in[0,1]$. We employ the
test function $\phi=y$ in \eqref{reg.form}.
This gives 
\begin{align}
  0 &=  \frac{2\sigma}{\alpha\tau} \int_{\T}e^{(1-\alpha/2)y}\left(\frac32 e^{\alpha y/2} - 2v_{k} 
  + \frac12 v_{k-1}\right)y \dx  \nonumber\\
  &\phantom{xx}{}
  + \frac{1}{2}\int_{\T}e^y (\pa_{ij}^2y)^2\dx 
  + \eps\int_{\T}\big((\Delta y)^2 + |\nabla y|^4 + y^2 \big)\dx. \label{reg.weak1}
\end{align}
To estimate the first integral, we distinguish the domains $\{y<0\}$ and $\{y\ge 0\}$:
\begin{align*}
  \frac{2\sigma}{\alpha\tau} \int_{\T}e^{(1-\alpha/2)y}&\left(\frac32 e^{\alpha y/2} - 2v_{k} 
  + \frac12 v_{k-1}\right)y \dx \\
  &= \frac{\sigma}{\alpha\tau}\int_{\{y < 0\}}\big(3e^{y}y - 4e^{(1-\alpha/2)y}v_{k}y 
  +  e^{(1-\alpha/2)y}v_{k-1}y\big)\dx \\ 
  &\phantom{xx}{} + \frac{\sigma}{\alpha\tau}\int_{\{y\geq 0\}}
  \big(3e^{y}y - 4e^{(1-\alpha/2)y}v_{k}y 
  +  e^{(1-\alpha/2)y}v_{k-1}y\big)\dx.
\end{align*}
The first integral on the right-hand side is estimated by using the Young
inequalities $-4e^{(1-\alpha/2)y}v_{k}y\ge -2e^{(2-\alpha)y} y^2 - 2v_k^2$ and
$e^{(1-\alpha/2)y}v_{k-1}y\ge -\frac12 e^{(2-\alpha)y} y^2 - \frac12v_{k-1}^2$:
\begin{align*}
  \frac{\sigma}{\alpha\tau}\int_{\{y < 0\}} & \big(3e^{y}y - 4e^{(1-\alpha/2)y}v_{k}y 
  +  e^{(1-\alpha/2)y}v_{k-1}y\big)\dx \\
  &\ge \frac{\sigma}{\alpha\tau}\int_{\{y < 0\}}\left(3e^y y - \frac52 e^{(2-\alpha)y} y^2 - 2v_k^2
  -\frac12 v_{k-1}^2\right)\dx \\
  &= \frac{\sigma}{\alpha\tau}\int_{\{y < 0\}}\left(e^y(y-1) + 1 
  + \left(1+2y\right)e^y -\frac52e^{(2-\alpha)y} y^2 - 1 - 2v_k^2 - \frac12v_{k-1}^2\right)\dx.
\end{align*}
Since $y\mapsto (1+2y)e^y -\frac52e^{(2-\alpha)y} y^2 - 1$, $y<0$, is bounded from below (remember that $\alpha<2$), we find that
\begin{align*}
  \frac{\sigma}{\alpha\tau}\int_{\{y < 0\}} &\big( 3e^{y}y - 4e^{(1-\alpha/2)y}v_{k}y 
  +  e^{(1-\alpha/2)y}v_{k-1}y\big)\dx\\
  &\ge \frac{\sigma}{\alpha\tau}\int_{\{y < 0\}} \big(e^y(y-1) + 1\big)\dx - \frac{\sigma}{\alpha\tau}c_1
  - \frac{\sigma}{\alpha\tau}\int_{\{y < 0\}}
  \left(2v_k^2 + \frac12 v_{k-1}^2\right)\dx, 
\end{align*}
where $c_1>0$ depends only on the lower bound of $y\mapsto (1+2y)e^y -\frac52 e^{(2-\alpha)y} y^2 - 1$, 
$y<0$, and $\mbox{meas}(\T)$. For the remaining integral over $\{y\ge 0\}$, 
we employ the Young inequalities $-4e^{(1-\alpha/2)y} v_k y\ge -2e^{(2-\alpha)y} - y^4 - v_k^4$ and
$e^{(1-\alpha/2)y}v_{k-1}y\ge -\frac12 e^{(2-\alpha)y} - \frac14 y^4 - \frac14 v_{k-1}^4$:
\begin{align*}
  \frac{\sigma}{\alpha\tau}\int_{\{y \ge 0\}} & \big(3e^{y}y - 4e^{(1-\alpha/2)y}v_{k}y 
  +  e^{(1-\alpha/2)y}v_{k-1}y\big)\dx \\
  &\ge \frac{\sigma}{\alpha\tau}\int_{\{y \ge 0\}}\Big(3e^y y - \frac52 e^{(2-\alpha)y} 
  - \frac54 y^4 - v_k^4 - \frac14 v_{k-1}^4\Big)\dx \\
  &= \frac{\sigma}{\alpha\tau}\int_{\{y \ge 0\}}\Big(e^y(y-1) + 1 + \Big((1+2y)e^y - \frac52 e^{(2-\alpha)y} - \frac54 y^4 - 1\Big) \\
&\phantom{xxxxxxxxxxx}{} - v_k^4 - \frac14 v_{k-1}^4\Big)\dx.
\end{align*}
The mapping $y\mapsto (1+2y)e^y - \frac52 e^{(2-\alpha)y} - \frac54 y^4 - 1$, $y\geq 0$, is bounded from below
which implies the existence of a constant $c_2>0$ such that
\begin{align*}
  \frac{\sigma}{\alpha\tau}\int_{\{y \ge 0\}} &\big(3e^{y}y - 4e^{(1-\alpha/2)y}v_{k}y 
  +  e^{(1-\alpha/2)y}v_{k-1}y\big)\dx\\
  &\ge \frac{\sigma}{\alpha\tau}\int_{\{y \ge 0\}}\big(e^y(y-1)+1\big)\dx
  - \frac{\sigma}{\alpha\tau}c_2 - \frac{\sigma}{\alpha\tau}\int_{\{y \ge 0\}}
  \left(v_k^4+\frac14 v_{k-1}^4\right)\dx.
\end{align*}
Summarizing the estimates for both integrals over $\{y>0\}$ and $\{y\ge 0\}$,
it follows that
\begin{align}
  \frac{2\sigma}{\alpha\tau} \int_{\T} & e^{(1-\alpha/2)y}\left(\frac32 e^{\alpha y/2} - 2v_{k} 
  + \frac12 v_{k-1}\right)y \dx \label{aux1} \ge \frac{\sigma}{\alpha\tau}\int_\T \big(e^y(y-1)+1\big)\dx\\
  &\phantom{xxxxxxx}{} - \frac{\sigma}{\alpha\tau}\int_\T
  \left(2v_k^2 + v_k^4 + \frac12 v_{k-1}^2 + \frac14 v_{k-1}^4\right)\dx
  - \frac{\sigma}{\alpha\tau}(c_1+c_2). \nonumber
\end{align}

For the second integral in \eqref{reg.weak1}, we use Lemma \ref{lem.H2}:
$$
  \frac12\int_{\T}e^y (\pa_{ij}^2y)^2\dx 
  \geq 2\kappa_1\int_{\T}\big(\Delta e^{y/2}\big)^2\dx,
$$
where $\kappa_1>0$ depends only on the space dimension $d$. With this estimate
and \eqref{aux1}, equation \eqref{reg.weak1} implies that
\begin{align*}
  \frac{\sigma}{\alpha\tau}\int_{\T} & \big(e^y(y-1)+1\big)\dx 
  + 2\kappa_1\int_{\T}\big(\Delta e^{y/2}\big)^2\dx 
  + \eps \int_{\T}\big((\Delta y)^2 + |\nabla y|^4 + y^2 \big)\dx \\ 
  &\leq \frac{\sigma}{\alpha\tau}\int_{\T}\left(2v_{k}^2 + v_{k}^4 
  + \frac12 v_{k-1}^2 + \frac14 v_{k-1}^4\right)\dx + \frac{\sigma}{\alpha\tau}(c_1 + c_2).
\end{align*} 
By the definition of the entropy, this inequality can be written as
\begin{align}
  E_1[n] &+ \frac{2\alpha\tau\kappa_1}{\sigma}\int_{\T}\big(\Delta e^{y/2}\big)^2\dx 
  + \frac{\eps\alpha\tau}{\sigma} \int_{\T} \big((\Delta y)^2 + |\nabla y|^4 + y^2\big)\dx \nonumber \\ 
  &\quad\leq \int_{\T}\left(2v_{k}^2 + v_{k}^4 + \frac12 v_{k-1}^2 
  + \frac14 v_{k-1}^4\right)\dx + c_1 + c_2.  \label{main.ap2}
\end{align}
The right-hand side gives a uniform (with respect to $\sigma$) bound since $v_{k-1}$, $v_k\in W^{1,4}(\T)$.
Hence, by the Poincar\'e inequality we obtain the $H^2$-bound
$$
  \|y\|_{H^2(\T)}^2 \le C\int_\T\big((\Delta y)^2 + y^2\big)\dx \le C,
$$
where the constant $C>0$ depends on $\alpha$, $\eps$, $\tau$, $v_k$,
and $v_{k-1}$ but not on $\sigma$. The continuous embedding $H^2(\T)\hookrightarrow
W^{1,4}(\T)$ then implies the desired uniform bound, $\|y\|_{W^{1,4}(\T)}
\le C$. Leray-Schauder's fixed-point theorem provides the existence of a fixed
point $y_\eps$ of $S(y,1)=y$, i.e.\ of a solution to \eqref{reg.form}.

{\em Step 3: Limit $\eps\to 0$.} Let $y_\eps$ be a solution to \eqref{reg.form},
constructed in the previous steps. Set $v_\eps:=e^{\alpha y_\eps/2}$ and $n_\eps:=e^{y_\eps}$. Then $v_\eps$ solves
\begin{equation}\label{veps}
  \frac{2}{\alpha\tau}v_\eps^{2/\alpha-1}
  \left(\frac32 v_\eps - 2v_k + \frac12 v_{k-1}\right)
  + \pa_{ij}^2\left(\frac{1}{\alpha}v_\eps^{2/\alpha-1}\pa_{ij}^2 v_\eps
  - \alpha \pa_i(v_\eps^{1/\alpha})\pa_j(v_\eps^{1/\alpha})\right)
  + \eps L(y_\eps) = 0.
\end{equation}
The goal is to pass to the limit $\eps\to 0$ in this equation.

Let $\alpha > 1$. We employ the test function $e^{(\alpha-1)y_\eps}/(\alpha-1)\in H^2(\T)$ 
in \eqref{reg.form} and find that
\begin{align*}
  0 &= \frac{2}{\alpha(\alpha-1)}\int_\T 
  \left(\frac32 v_\eps - 2v_k + \frac12 v_{k-1}\right)v_\eps\dx
  + \frac{\tau}{2(\alpha-1)}\int_\T e^{y_\eps}\pa_{ij}^2 y_\eps\pa_{ij}^2
  (e^{(\alpha-1)y_\eps})\dx \\
  &\phantom{xx}{}
  + \frac{\eps\tau}{\alpha-1}\langle L(y_\eps),e^{(\alpha-1)y_\eps}\rangle_{H^{-2},H^2}.
\end{align*}
Inequality \eqref{ineq1} shows that
\begin{align*}
  \frac{2}{\alpha(\alpha-1)}\int_\T 
  \left(\frac32 v_\eps - 2v_k + \frac12 v_{k-1}\right)v_\eps\dx
  &\ge \frac{1}{\alpha(\alpha-1)}\int_\T\left(\frac32 v_\eps^2 - 2v_k^2 
  + \frac12 v_{k-1}^2\right)\dx \\
  &\phantom{xx}{}+ \frac{1}{\alpha(\alpha-1)}\int_\T\big((v_\eps-v_k)^2
  - (v_k-v_{k-1})^2\big)\dx.
\end{align*}
The integral involving the second derivatives is again estimated by using
Lemma \ref{lem.H2}:
$$
  \frac{\tau}{2(\alpha-1)}\int_\T e^{y_\eps}\pa_{ij}^2 y_\eps\pa_{ij}^2
  (e^{(\alpha-1)y_\eps})\dx
  \ge 2 \kappa_\alpha\tau\int_\T(\Delta e^{\alpha y_\eps/2})^2 \dx
  = 2\kappa_\alpha\tau\int_\T(\Delta v_\eps)^2 \dx.
$$
Now let us consider the $\eps$-term and show that $\langle L(y_\eps),e^{(\alpha-1)y_\eps}\rangle_{H^{-2},H^2}$ is bounded from below uniformly in $\eps$.
By construction, $v_\eps$ and $n_\eps$ are strictly positive
since $y_\eps\in H^2(\T)\hookrightarrow L^\infty(\T)$. Therefore, we can write (cf.~\cite[Section 4.1]{JuMa08})
\begin{align*}
\langle L(y_\eps),e^{(\alpha-1)y_\eps}&\rangle_{H^{-2},H^2} = 4(\alpha-1)\int_{\T}\Big(e^{(\alpha-1)y_\eps}\Big(\frac{\Delta e^{y_\eps/2}}{e^{y_\eps/2}} - (2-\alpha)\Big|\frac{\nabla e^{y_\eps/2}}{e^{y_\eps/2}}\Big|^2\Big)^2\dx\\ 
&\phantom{xx}{} + 4(\alpha^2-1)(3-\alpha)\int_{\T}e^{(\alpha-1)y_\eps} \Big|\frac{\nabla e^{y_\eps/2}}{e^{y_\eps/2}}\Big|^4\dx 
 + \int_{\T}y_\eps e^{(\alpha-1)y_\eps}\dx \geq -C,
\end{align*} 
where $C>0$ depends only on $\alpha$. We have used the fact that $x e^{(\alpha-1)x}\geq -1/((\alpha-1)e)$ for all $x\in\R$.

Summarizing the above inequalities, we obtain
\begin{align}
  \frac{1}{\alpha(\alpha-1)}\int_\T & \left(\frac32 v_\eps^2 - 2v_k^2 
  + \frac12 v_{k-1}^2\right)\dx
  + \frac{1}{\alpha(\alpha-1)}\int_\T\big((v_\eps-v_k)^2 - (v_k-v_{k-1})^2\big)\dx\nonumber \\
  &{}+ 2\tau\kappa_\alpha\int_\T(\Delta v_\eps)^2 \dx
  \le C\eps.\label{main.ap1}
\end{align}
Inequality \eqref{main.ap1} provides the estimate for
$(v_\eps)$ in $H^2(\T)$ uniformly in $\eps$. Therefore, there exists a limit
function $v\in H^2(\T)$ such that, up to a subsequence, as $\eps\to 0$,
\begin{align*}
  v_\eps \rightharpoonup v &\quad\mbox{weakly in }H^2(\T), \\
  v_\eps\to v &\quad\mbox{strongly in }W^{1,4}(\T)\mbox{ and }L^\infty(\T).
\end{align*}
Consequently, since $2/\alpha-1>0$,
\begin{equation}\label{conv1}
  v_\eps^{2/\alpha-1}\pa_{ij}^2 v_\eps \rightharpoonup 
  v^{2/\alpha-1}\pa_{ij}^2 v \quad\mbox{weakly in } L^2(\T),\ i,j=1,\ldots,d.
\end{equation}

According to the Lions-Villani lemma on the regularity
of the square root of Sobolev functions (see the version in \cite[Lemma 26]{BJM11}), 
there exists $C>0$ independent of $\eps$ such that 
$$
  \|\sqrt{v_\eps}\|_{W^{1,4}(\T)}^2 \le C\|v_\eps\|_{H^2(\T)}\le C.
$$
Since $1/2 < 1/\alpha < 1$, Proposition A.1 in \cite{JuMi09} shows that
the strong convergence $v_\eps\to v$ in $H^1(\T)$ and the boundedness of
$(\sqrt{v_\eps})$ in $W^{1,4}(\T)$ imply that
$$
  v_\eps^{1/\alpha} \to v^{1/\alpha} \quad\mbox{strongly in }W^{1,2\alpha}(\T).
$$
Hence, we have
\begin{equation}\label{conv2}
  \pa_i(v_\eps^{1/\alpha})\pa_j(v_\eps^{1/\alpha})\to 
  \pa_i(v^{1/\alpha})\pa_j(v^{1/\alpha}) \quad\mbox{strongly in }
  L^\alpha(\T),\ i,j=1,\ldots,d.
\end{equation}
Estimate \eqref{main.ap2} and $E_1[n]\ge0$ provide the uniform bound
$$
  \sqrt{\eps}\|y_\eps\|_{H^2(\T)} + \sqrt[4]{\eps}\|\na y_\eps\|_{L^4(\T)} 
\le C,
$$
which shows that
\begin{equation}\label{conv3}
  \eps L(y_\eps)\rightharpoonup 0 \quad\mbox{ weakly in }H^{-2}(\T).
\end{equation}

Using $\phi\in W^{2,\infty}(\Omega)$ as a test function in the weak formulation
of \eqref{veps}, the convergence results \eqref{conv1}-\eqref{conv3} allow us to 
pass to the limit $\eps\to 0$ in the resulting equation, which yields 
\eqref{weak_alpha} for $v_{k+1}:=v$. In fact, it is sufficient to use test functions
$\phi\in W^{2,\alpha/(\alpha-1)}(\T)$. 

If $\alpha=1$, the convergence result follows similarly as above based on the uniform bound $\|e^{y_\eps/2}\|_{H^2}\leq C$, which is obtained from a priori estimate \eqref{main.ap2}, using the elementary inequality $s \leq s(\log s - 1) + e$ for all $s \geq 0$, which gives a uniform $L^2$-bound for $e^{y_\eps}$. 
In that case, the test functions $\phi\in H^2(\T)$ can be used in \eqref{veps}.

{\em Step 4: Entropy stability.} 
Let $\alpha>1$. Using the test function $v_1^{2-2/\alpha}/(\alpha-1)$ in \eqref{alpha_euler}, 
it follows that
\begin{equation*}
  \frac{1}{\tau\alpha(\alpha-1)}\int_\T\big(v_1^2 - v_0^2 + (v_1-v_0)^2\big)\dx
  + \frac{1}{2(\alpha-1)}\int_\T v_1^{2/\alpha}\pa_{ij}^2(\log v_1^{2/\alpha})
  \pa_{ij}^2(v_1^{2-2/\alpha})\dx = 0.
\end{equation*}
By Lemma \ref{lem.H2}, we infer that
\begin{equation}\label{5.1}
  \frac{1}{\alpha(\alpha-1)}\int_\T \big(v_1^2 + (v_1-v_0)^2\big) \dx 
  + 2\tau\kappa_\alpha\int_\T(\Delta v_1)^2\dx
  \le \frac{1}{\alpha(\alpha-1)}\int_\T v_0^2 \dx.
\end{equation}
This gives an $H^2$-bound for $v_1$. 

Next, let $k\ge 1$ and let
$y_\eps$ be a weak solution to \eqref{reg.form}. Set $v_\eps=e^{\alpha y_\eps/2}$. 
The convergence results of Step 3 allow us to pass to the limit $\eps\to 0$
in \eqref{main.ap1}.
Using the weakly lower semi-continuity of $u\mapsto\|\Delta u\|_{L^2(\T)}^2$
on $H^2(\T)$, it follows that
\begin{align}
  \frac{1}{\alpha(\alpha-1)}\int_\T & \left(\frac32 v_{k+1}^2 - 2v_k^2 
  + \frac12 v_{k-1}^2\right)\dx
  + \frac{1}{\alpha(\alpha-1)}\int_\T\big((v_{k+1}-v_k)^2 - (v_k-v_{k-1})^2\big)\dx 
  \nonumber \\
  &{}+ 2\kappa_\alpha\tau\int_\T(\Delta v_{k+1})^2 \dx \le 0,  \label{5.k}
\end{align}
where, as before, $v_{k+1}=\lim_{\eps\to 0}v_\eps$. Summing
\eqref{5.1} and \eqref{5.k} over $k=1,\ldots,m-1$, some terms cancel and we
end up with
$$
  \frac{3}{2\alpha(\alpha-1)}\int_\T v_m^2 \dx
  + 2\kappa_\alpha\tau\sum_{k=0}^{m-1}\int_\T(\Delta v_{k+1})^2\dx
  \le \frac{1}{2\alpha(\alpha-1)}\int_\T(v_{m-1}^2 + v_1^2 + v_0^2)\dx.
$$
Set $a_m=\|v_m\|_{L^2(\Omega)}^2$ for $m\ge 0$. By \eqref{5.1}, $a_1\le a_0$.
Then, the above estimate shows that $a_m\le \frac13 a_{m-1} + \frac23 a_0$.
A simple induction argument gives $a_m\le a_0$ for all $m\ge 1$.
Therefore, 
$$
  \frac{1}{\alpha(\alpha-1)}\int_\T v_m^2\dx 
  + \frac43\kappa_\alpha\tau\sum_{k=1}^m\int_\T(\Delta v_k)^2\dx
  \le \frac{1}{\alpha(\alpha-1)}\int_\T v_0^2 \dx.
$$
This implies the entropy stability estimate \eqref{ent.stab}.
\end{proof}

The proof of Corollary \ref{coro.bdf2} is a consequence of the above proof.
Indeed, employing inequality \eqref{ineq2} instead of \eqref{ineq1}, we can
replace \eqref{5.k} by
\begin{align*}
  \frac{1}{2\alpha(\alpha-1)}\int_\T & \big(v_{k+1}^2 + (2v_{k+1}-v_k)^2\big)\dx
  + 2\kappa_\alpha\tau\int_\T(\Delta v_{k+1})^2 \dx \\
  &\le \frac{1}{2\alpha(\alpha-1)}\int_\T \big(v_{k}^2 + (2v_{k}-v_{k-1})^2\big)\dx,
\end{align*}
which equals \eqref{ent.diss}.


Next, we prove that, if $\alpha=1$, 
the solutions $v_k$ are smooth as long as they are strictly positive.

\begin{lemma}\label{lem.smooth}
Let $\alpha=1$ and let $(v_k)$ be the sequence of weak solutions constructed
in Theorem~\ref{thm.bdf2.ex} satisfying $v_k\ge \mu_k>0$ in $\T$ for $k\ge 1$
and some $\mu_k>0$. Then $v_k\in C^\infty(\T)$.
\end{lemma}

\begin{proof}
We recall that the weak form \eqref{weak_alpha} for $\alpha=1$ reads as
$$
  \int_\T v_{k+1}\left(\frac32 v_{k+1}-2v_k+\frac12 v_{k-1}\right)\phi \dx
  + \frac{\tau}{2}\int_\T\big(v_{k+1}\pa_{ij}^2v_{k+1}-\pa_i v_{k+1}\pa_j v_{k+1}\big)
  \pa_{ij}^2\phi \dx = 0
$$
for $\phi\in H^2(\T)$.
Since $v_k$ is assumed to be strictly positive, we can write
$$
  v_{k+1}\pa_{ij}^2v_{k+1}-\pa_i v_{k+1}\pa_j v_{k+1} 
  = \frac12 n_{k+1}\pa_{ij}^2\log n_{k+1},
$$
where $n_{k+1}=v_{k+1}^2$ and consequently, 
\begin{equation}\label{vk}
  v_{k+1}\left(\frac32 v_{k+1}-2v_k+\frac12 v_{k-1}\right)
  + \frac{\tau}{4}\pa_{ij}^2(n_{k+1}\pa_{ij}^2\log n_{k+1}) = 0
  \quad\mbox{in }H^{-2}(\T).
\end{equation}
With the identity
$$
  \pa_{ij}^2(n_{k+1}\pa_{ij}^2\log n_{k+1}) 
  = \Delta^2 n_{k+1} - \pa_i\left(2\frac{\pa_{ij}^2n_{k+1}\pa_j n_{k+1}}{n_{k+1}}
  - \frac{(\pa_j n_{k+1})^2\pa_i n_{k+1}}{n_{k+1}^2}\right),
$$
it follows that $n_{k+1}$ solves
\begin{equation}\label{delta2}
  \Delta^2 n_{k+1} = \pa_i\left(2\frac{\pa_{ij}^2n_{k+1}\pa_j n_{k+1}}{n_{k+1}}
  - \frac{(\pa_j n_{k+1})^2\pa_i n_{k+1}}{n_{k+1}^2}\right)
  - \frac{4}{\tau}v_{k+1}\left(\frac32 v_{k+1}-2v_k+\frac12 v_{k-1}\right)
\end{equation}
in the sense of $H^{-2}(\T)$. The second term on the right-hand side is an
element of $H^2(\T)$. The continuity of the Sobolev embedding
$H^2(\T)\hookrightarrow W^{1,6}(\T)$ (for $d\le 3$) implies that
$(\pa_j n_{k+1})^2\pa_i n_{k+1}/n_{k+1}\in L^2(\T)$ and
$\pa_{ij}^2n_{k+1}\pa_j n_{k+1}/n_{k+1}\in L^{3/2}(\T)\hookrightarrow H^{-1/2}(\T)$
for all $i,j=1,\ldots,d$. This proves that
$$
  \Delta^2 n_{k+1}\in H^{-3/2}(\T).
$$
The regularity theory for elliptic operator on $\T$ (e.g., using Fourier
transforms on the torus) yields $n_{k+1}\in H^{5/2}(\T)$ which improves
the previous regularity $n_{k+1}\in H^2(\T)$. Taking into account the improved
regularity and the embedding $H^{5/2}(\T)\hookrightarrow W^{2,3}(\T)$,
we infer that the right-hand side of \eqref{delta2} lies in $H^{-1}(\T)$, i.e.
$$
  \Delta^2 n_{k+1}\in H^{-1}(\T),
$$
which implies that $n_{k+1}\in H^3(\T)$. By bootstrapping, we conclude that
$n_{k+1}\in H^m(\T)$ for all $m\in\N$.
\end{proof}

Now, we are in the position to prove Theorem \ref{thm.bdf2.conv}.

\begin{proof}[Proof of Theorem \ref{thm.bdf2.conv}]
Let $(v_k)$ be a sequence of weak solutions to \eqref{weak_alpha}.
Since we have assumed that $v_k$ is strictly positive, Lemma \ref{lem.smooth}
shows that $v_k$ is smooth. As a consequence, $v_k$ solves (see \eqref{vk})
$$
  \frac32 v_{k+1} - 2v_k + \frac12 v_{k-1} + \frac{1}{v_{k+1}}
  \pa_{ij}^2\big(v_{k+1}^2\pa_{ij}^2\log v_{k+1}\big) = 0 \quad\mbox{in }\T.
$$
Let $n=v^2$ be a solution to \eqref{dlss} with the regularity indicated in the
theorem. By Taylor expansion,
$$
  v_t(t_{k+1}) = \frac{1}{\tau}\left(\frac32 v(t_{k+1}) - 2v(t_k) 
  + \frac12 v(t_{k-1})\right) + \frac{f_k}{\tau}, \quad k\ge 1,
$$
where
$$
  f_k = -\int_{t_k}^{t_{k+1}} v_{ttt}(s)(t_k-s)^2 \dd s
  + \frac14\int_{t_{k-1}}^{t_{k+1}}v_{ttt}(s)(t_{k-1}-s)^2\dd s
$$
can be interpreted as the local truncation error. We estimate $f_k$ as follows:
\begin{equation}\label{f.k}
  \sum_{k=1}^{m-1}\|f_k\|_{L^2(\T)}^2 \le C_R\|v_{ttt}\|_{L^2(0,T;L^2(\T))}^2\tau^5,
\end{equation}
where $C_R>0$ does not depend on $\tau$ or $m$. Similarly, we have
$$
  v_t(t_1) = \frac{1}{\tau}(v(t_1)-v(t_0)) + \frac{f_0}{\tau}, \quad\mbox{where }
  f_0 = \int_0^\tau v_{tt}(s)s\dd s,
$$
and
\begin{equation}\label{f.0}
  \|f_0\|_{L^2(\T)} 
  \le \int_0^\tau \|v_{tt}(s)\|_{L^2(\T)} s \dd s
  \le \frac{\tau^2}{2}\|v_{tt}\|_{L^\infty(0,T;L^2(\T))}.
\end{equation}
Replacing the time derivative $v_t$ in \eqref{dlss}, written as
$v_t + v^{-1}\pa_{ij}^2(v^2\pa_{ij}^2\log v) = 0$,
by the above expansions, it follows that
\begin{align}
  v(t_1)-v(t_0) + \frac{\tau}{v(t_1)}\pa_{ij}^2\big(v(t_1)^2\pa_{ij}^2\log v(t_1)\big)
  &= -f_0, \label{v.0} \\
  \frac32 v(t_{k+1}) - 2v(t_k) + \frac12 v(t_{k-1}) + \frac{\tau}{v(t_{k+1})}
  \pa_{ij}^2\big(v(t_{k+1})^2\pa_{ij}^2\log v(t_{k+1})\big)
  &= -f_k, \label{v.k}
\end{align}
for $k\ge 1$. Taking the difference of \eqref{alpha_euler}, multiplied by $v_1^{-1}$, 
and \eqref{v.0}, and the difference of \eqref{disc.dlss}, multiplied by $v_{k+1}^{-1}$, 
and \eqref{v.k}, we obtain the error equations for $e_k:=v_k-v(t_k)$:
\begin{align*}
  e_1-e_0 + \tau\big(A(v_1)-A(v(t_1))\big) &= f_0, \\
  \frac32 e_{k+1} - 2e_k + \frac12 e_{k-1} 
  + \tau\big(A(v_{k+1})-A(v(t_{k+1}))\big) &= f_k, \quad k\ge 1,
\end{align*}
where we have introduced the operator 
$$
  A:D(A)\to H^{-2}(\T), \quad A(v) = \frac{1}{v}\pa_{ij}^2(v^2\pa_{ij}^2\log v),
$$
with domain $D(A)=\{v\in H^2(\T):v>0$ in $\T\}$. 

We multiply the error equations by $e_1$ and $e_{k+1}$, respectively, integrate
over $\T$, and sum over $k=0,\ldots,m-1$:
\begin{align}
  \int_\T & (e_1-e_0)e_1 \dx + \sum_{k=1}^{m-1}\int_\T
  \left(\frac32 e_{k+1}-2e_k+\frac12 e_{k-1}\right)e_{k+1}\dx \nonumber \\
  &{}+ \tau\sum_{k=0}^{m-1}\int_\T\big(A(v_{k+1})-A(v(t_{k+1}))\big)
  (v_{k+1}-v(t_{k+1}))\dx = \sum_{k=0}^{m-1}f_k e_{k+1}\dx. \label{err}
\end{align}
Using $e_0=0$ and inequality \eqref{ineq1}, the first two integrands can be 
estimated by
\begin{align*}
  (e_1-e_0)e_1 +& \sum_{k=1}^{m-1}
  \left(\frac32 e_{k+1}-2e_k+\frac12 e_{k-1}\right)e_{k+1} \\
  &\ge e_1^2 + \sum_{k=1}^{m-1}\left(\frac34 e_{k+1}^2 - e_k^2 + \frac14 e_{k-1}^2
  + \frac12(e_{k+1}-e_k)^2 - \frac12(e_k-e_{k-1})^2\right)\dx \\
  &= e_1^2 + \frac34 e_m^2 - \frac34 e_1^2 - \frac14e_{m-1}^2 + \frac14 e_0^2
  + \frac12(e_m-e_{m-1})^2 - \frac12(e_1-e_0)^2 \\
  &= \frac34 e_m^2 - \frac14 e_{m-1}^2 - \frac14 e_1^2 + \frac12(e_m-e_{m-1})^2 \\
  &\ge \frac34 e_m^2 - \frac14 e_{m-1}^2 - \frac14 e_1^2.
\end{align*}
For the third integral in \eqref{err}, we employ the monotonicity of the
operator $A$. In fact, it is proved in \cite[Lemma 3.5]{JuPi03} that for
positive functions $w_1$, $w_2\in H^4(\T)$,
$$
  \int_\T(A(w_1)-A(w_2))(w_1-w_2)\dx = \int_\T\frac{1}{w_1w_2}
  \left|\diver\left(w_1^2\na\left(\frac{w_1-w_2}{w_1}\right)\right)\right|^2\dx \ge 0.
$$
The right-hand side of \eqref{err} is estimated by Young's inequality:
\begin{align*}
  \int_{\T}f_0 e_{1}\dx 
  &\le 2\|f_0\|_{L^2(\T)}^2 + \frac18\|e_1\|_{L^2(\T)}^2, \\ 
  \int_{\T}f_k e_{k+1}\dx 
  &\le \frac{1}{2\tau}\|f_k\|_{L^2(\T)}^2 
  + \frac{\tau}{2}\|e_{k+1}\|_{L^2(\T)}^2,\quad k\geq 1.
\end{align*}
Summarizing the above estimates and taking into account \eqref{f.k} and \eqref{f.0}, 
we find that
\begin{align*}
  \frac34\|e_m\|_{L^2(\T)}^2
  &\le \frac14\|e_{m-1}\|_{L^2(\T)}^2 + \frac14\|e_1\|_{L^2(\T)}^2
  + 2\|f_0\|_{L^2(\T)}^2 + \frac18\|e_1\|_{L^2(\T)}^2 \\
  &\phantom{xx}{}+ \frac{1}{2\tau}\sum_{k=1}^{m-1}\|f_k\|_{L^2(\T)}^2
  + \frac{\tau}{2}\sum_{k=1}^{m-1}\|e_{k+1}\|_{L^2(\T)}^2 \\
  &\le \frac14\|e_{m-1}\|_{L^2(\T)}^2 + \frac38\|e_1\|_{L^2(\T)}^2
  + C\tau^4 + \frac{\tau}{2}\sum_{k=2}^{m}\|e_{k}\|_{L^2(\T)}^2,
\end{align*}
where $C>0$ depends on the $L^2(0,T;L^2(\T))$ norm of $v_{ttt}$ and the
$L^\infty(0,T;L^2(\T))$ norm of $v_{tt}$ but not on $\tau$.
Taking the maximum over $m=1,\ldots,M$, we infer that
$$
  \frac34\max_{m=1,\ldots,M}\|e_m\|_{L^2(\T)}^2
  \le \frac58\max_{m=1,\ldots,M}\|e_{m-1}\|_{L^2(\T)}^2
  + C\tau^4 + \frac{\tau}{2}\sum_{k=2}^{M}\|e_{k}\|_{L^2(\T)}^2.
$$
The first term on the right-hand side is controlled by the left-hand side, 
leading to
$$
  \|e_M\|_{L^2(\T)}^2 \le \max_{m=1,\ldots,M}\|e_m\|_{L^2(\T)}^2
  \le 8C\tau^4 + 4\tau\sum_{k=2}^{M}\|e_{k}\|_{L^2(\T)}^2.
$$
We separate the last summand in the sum,
$$
  (1-4\tau)\|e_M\|_{L^2(\T)}^2 
  \le 8C\tau^4 + 4\tau\sum_{k=2}^{M-1}\|e_{k}\|_{L^2(\T)}^2,
$$
and apply the inequality $1+x\le e^x$ for all $x\ge 0$ and
the discrete Gronwall lemma (see, e.g., \cite[Theorem 4]{WiWo65}):
\begin{align*}
  \|e_M\|_{L^2(\T)}^2 &\le \frac{8C\tau^4}{1-4\tau}
  \left(1+\frac{4\tau}{1-4\tau}\right)^{M-2}
  \leq \frac{8C\tau^4}{1-4\tau}\exp\Big(\frac{4t_{M-2}}{1-4\tau}\Big) 
  \le 16C\tau^4 \exp(8t_{M-2}).
\end{align*}
The result follows for all $0<\tau<1/8$ with the constant $4\sqrt{C}\exp(4T)$,
where $T>0$ is the terminal time.
\end{proof}


\section{Fully discrete variational derivative method}\label{sec.dvdm}

In this section, we explore the variational structure of the DLSS equation
on a discrete level, using the discrete variational derivative method of
\cite{FuMa10}. In order to explain the idea, we consider first the implicit Euler
discretization. 

Let $x_i=ih$, $i=0,\ldots,N-1$, be an equidistant grid on the
one-dimensional torus $\K\cong[0,1)$, let $t_k=k\tau$ with $\tau>0$, 
and let $U_i^k$ approximate $n(t_k,x_i)$.
Set $U^k=(U_0^k,\ldots,U_{N-1}^k)\in\R^N$ and $U_\ell = U_{\ell\mod N}$ for all 
$\ell\in\Z$. 
We introduce the following difference operators for $U=(U_i)\in\R^N$:

\begin{center}
\begin{tabular}{ll}
  forward difference: & $\delta_i^+ U = h^{-1}(U_{i+1}-U_i)$, \\[1mm]
  backward difference: & $\delta_i^- U = h^{-1}(U_i-U_{i-1})$, \\[1mm]
  central difference: & $\delta^{\langle1\rangle}_i U =
  (2h)^{-1}(U_{i+1}-U_{i-1})$, \\[1mm] second-order central difference: & 
  $\delta^{\langle2\rangle}_i U = \delta_i^+\delta_i^- U = \delta_i^-\delta_i^+ U$.
\end{tabular}
\end{center}

The first step is to define the discrete Fisher information.
We choose a symmetric form for the derivative, $v_x^2(x_i)\approx
\frac12((\delta_i^+ V)^2+(\delta^-_i V)^2)$, where $V=(V_i)=(\sqrt{U_i})\in\R^N$.
The Fisher information $F[v^2]=\int_\K v_x^2 \dx$ is approximated by using
the first-order quadrature rule $\int_\K w(x)\dx\approx \sum_{i=0}^{N-1}w(x_i)h$.
Actually, this rule is of second order $O(h^2)$ here, 
since due to the periodic boundary conditions, 
it coincides with the trapezoidal rule, $(w(x_0)+w(x_N))h/2+\sum_{I=1}^{N-1}w(x_i)h$.
Therefore, the discrete Fisher information reads as
$$
  F_d[U] = \frac{1}{2}\sum_{i=0}^{N-1}\big((\delta_i^+ V)^2 + (\delta_i^- V)^2\big)h,
  \quad U=(U_i)\in\R^N.
$$

The second step is the definition of the discrete variational derivative.
Applying the discrete variation procedure and using summation by parts
(see \cite[Prop.\ 3.2]{FuMa10}), we calculate
\begin{align}
  F_d[U^{k+1}] - F_d[U^k] 
  &= \frac12\sum_{i=0}^{N-1}\left((\delta_i^+V^{k+1})^2 - (\delta_i^+V^k)^2 
  + (\delta_i^-V^{k+1})^2 - (\delta_i^-V^k)^2\right)h \nonumber \\
  &= \frac12\sum_{i=0}^{N-1}\big[\delta_i^+(V^{k+1} + V^k)\delta_i^+(V^{k+1} - V^k) 
  \nonumber \\
  &\phantom{xx}{}+ \delta_i^-(V^{k+1} + V^k)\delta_i^-(V^{k+1} - V^k)\big]h 
  \nonumber \\
  &= -\sum_{i=0}^{N-1}\delta_i^{\langle 2\rangle}(V^{k+1} + V^k)(V_i^{k+1} -
  V_i^k)h \nonumber \\
  &= -\sum_{i=0}^{N-1}\frac{\delta_i^{\langle 2\rangle}(V^{k+1} + V^k)}
  {V_i^{k+1} + V_i^k}(U_i^{k+1} - U_i^k)h, \quad k\geq 0. \label{dcr}
\end{align}
This motivates the definition of the discrete variational derivative
\begin{equation}\label{q1.dvd}
  \frac{\delta F_d}{\delta(U^{k+1},U^k)_i} 
  = -\frac{\delta_i^{\langle 2\rangle}(V^{k+1} + V^k)}{V_i^{k+1} + V_i^k}, \quad   
  i=0,\ldots,N-1,
\end{equation}
since this implies the discrete chain rule
$$
  F_d[U^{k+1}]-F_d[U^k] = \sum_{i=0}^{N-1}\frac{\delta F_d}{\delta(U^{k+1},U^k)_i}
  (U_i^{k+1}-U_i^k)h.
$$
Observe that \eqref{q1.dvd} is a Crank-Nicolson type approximation of
the variational derivative 
$\delta F[n]/\delta n=-(\sqrt{n})_{xx}/\sqrt{n}=-v_{xx}/v$, where $n=v^2$.
The implicit Euler discrete variational derivative (DVD) method for the DLSS
equation is then given by the nonlinear system with unknowns 
$U^{k+1} = (V^{k+1})^2$:
\begin{equation}\label{bdf1.dvdm}
  \frac{1}{\tau}(U_i^{k+1} - U_i^k) 
  = \delta_i^{\langle 1\rangle}\left(U^{k+1}\delta_i^{\langle 1\rangle}
  \left(\frac{\delta F_d}{\delta(U^{k+1},U^k)}\right)\right), \quad
  i=0,\ldots,N-1,\ k\geq 0.
\end{equation}
The initial condition $n_0$ is approximated by its projection on the discrete grid,
defining the starting vector $U^0\in\R^N$. 
Multiplying the above scheme by $\delta F_d/\delta(U^{k+1},U^k)_i$, summing
over $i=0,\ldots,N-1$, and employing the discrete chain rule \eqref{dcr},
we infer the discrete dissipation property
\begin{equation}\label{dvd.diss}
  \frac{1}{\tau}(F_d[U^{k+1}] - F_d[U^k])
  + \sum_{i=0}^{N-1}U_i^{k+1}\left(\delta_i^{\langle 1\rangle}
  \left(\frac{\delta F_d}{\delta(U^{k+1},U^k)}\right)\right)^2h = 0.
\end{equation}
In fact, this proves the monotonicity of the discrete Fisher information for $q=1$. 

\begin{remark}\rm
Observe that we could have taken a different approximation for the
discrete Fisher information, e.g.\ $\widetilde{F}_d[U] =
\sum_{i=0}^{N-1}(\delta_i^{\langle 1\rangle} V)^2h$. This would lead to a different variational
derivative $\delta \widetilde{F}_d/\delta(U^{k+1},U^k)$ and eventually to a
another scheme (\ref{bdf1.dvdm}), with $F_d$ replaced by $\widetilde{F}_d$,
which dissipates $\widetilde{F}_d$ instead. Besides the symmetry, which brings the second-order consistency in space, the above choice of the
discrete Fisher information is motivated by the fact that $\delta_i^+\delta_i^- = \delta_i^{\langle 2\rangle}$, used in the discrete variation procedure.
\hfill\qed
\end{remark}

In the following, we consider temporally higher-order discretizations.
There are several ways to generalize the above DVD method. In order to stay in the
spirit of Section \ref{sec.bdf2}, we derive higher-order DVD methods, which are based
on backward differentiation formulas. The function $f(\xi,\eta)=(\xi^2+\eta^2)/2$
represents both the Fisher information $F[n]=\int_\K f(v_x,v_x)\dx$ and
the discrete Fisher information $F_d[U]=\sum_{i=0}^{N-1}f(\delta_i^+V,\delta_i^-V)h$.
The definition of $f$ is motivated by the following formal representation
of the variational derivative,
$$
  \frac{\delta F[n]}{\delta n} 
  = -\frac{v_{xx}}{v} 
  = -\frac{1}{2v}\left(\pa_x\pa_\xi f\big|_{\xi=v_x}
  + \pa_x\pa_\eta f\big|_{\eta=v_x}\right).
$$
This formula gives an idea how to approximate the variational derivative in general.
We denote by $\delta^{1,q}_k$ the $q$-th step BDF operator at time $t_k$.
For instance, the formulas for $q=1$ and $q=2$ are given in \eqref{1.bdf1} and
\eqref{1.bdf2}, respectively.
The discrete variational derivative of order $q$ is defined componentwise by
\begin{equation}\label{gen.dvd}
  \frac{\delta F_d}{\delta(U^{k+1},\ldots,U^{k-q+1})_i} 
  = -\frac{1}{2V_i^{k+1}}\left(\delta_i^-(\pa_\xi^{\dd} f) 
  + \delta_i^+(\pa_\eta^{\dd} f)\right), \quad k\geq q-1,
\end{equation}
where the discrete operators $\pa^\dd_\xi f$ and $\pa_\eta^\dd f$ are given by
\begin{align*}
  (\pa_\xi^{\dd} f)_i 
  &= \pa_\xi f\big|_{\xi = \delta_i^+V^{k+1}} 
  + r_{\rm corr}\delta_{k+1}^{1,q}(\delta_i^+U^{k+1}) 
  = \delta_i^+V^{k+1} + r_{\rm corr}\delta_{k+1}^{1,q}(\delta_i^+U^{k+1}), \\
  (\pa_\eta^{\dd} f)_i 
  &= \pa_\eta f\big|_{\eta = \delta_i^-V^{k+1}} 
  + r_{\rm corr}\delta_{k+1}^{1,q}(\delta_i^-U^{k+1}) 
  = \delta_i^-V^{k+1} + r_{\rm corr}\delta_{k+1}^{1,q}(\delta_i^-U^{k+1}),
\end{align*}
and $r_{\rm corr}$ is a correction term, which has to be determined in such a way
that the discrete chain rule
$$
  \delta_{k+1}^{1,q} F_d[U^{k+1}] 
  = \sum_{i=0}^{N-1}\frac{\delta F_d}{\delta(U^{k+1},\ldots,U^{k-q+1})_i}
  \delta_{k+1}^{1,q}U_i^{k+1}h
$$
holds.
The role of the correction term is not only to satisfy the discrete chain rule
but also to increase the temporal accuracy of the discrete variational
derivative. Straightforward computations with the above expressions using
summation by parts formulas and periodic boundary conditions yield
\begin{align}
  & \frac{\delta F_d}{\delta(U^{k+1},\ldots,U^{k-q+1})_i} 
  = -\frac{\delta_i^{\langle2\rangle}V^{k+1}}{V_i^{k+1}} 
  - r_{\rm corr}
  \frac{\delta_{k+1}^{1,q}\delta_i^{\langle2\rangle}U^{k+1}}{V_i^{k+1}}, \quad
  k\geq q-1, \label{dvd.bdfq} \\
  & r_{\rm corr} 
  = \frac{\delta_{k+1}^{1,q}F_d[U^{k+1}] - \sum_{i=0}^{N-1}\delta_i^+V^{k+1}
  \delta_i^+\Big(\frac{\delta_{k+1}^{1,q}U^{k+1}}{V^{k+1}}\Big)h}{\sum_{i=0}^{N-1}
  (\delta_i^+\delta_{k+1}^{1,q}U^{k+1})
  \delta_i^+\Big(\frac{\delta_{k+1}^{1,q}U^{k+1}}{V^{k+1}}\Big)h}. \label{rcorr}
\end{align}
We note that for $q=1$, this definition generally does not coincide with
the discrete variational derivative \eqref{q1.dvd}. 
The temporally BDF$q$ discrete variational derivative (BDF$q$ DVD) method
is then defined by the following nonlinear system in the unknowns
$U^{k+1}=(V^{k+1})^2$:
\begin{equation}\label{bdfq.dvdm}
  \delta_{k+1}^{1,q}U_i^{k+1} 
  = \delta_i^{\langle1\rangle}\left(U^{k+1}\delta_i^{\langle1\rangle}
  \left(\frac{\delta F_d}{\delta(U^{k+1},\ldots,U^{k-q+1})}\right)\right), \quad
  i=0,\ldots,N-1,\ k\geq q-1.
\end{equation}
In particular, for $q=1$, we obtain two methods: the BDF1 DVD scheme \eqref{bdfq.dvdm}
and the DVD scheme \eqref{bdf1.dvdm}.

\begin{proof}[Proof of Theorem \ref{thm.dvdm}.]
Let $n=v^2$ be a smooth positive solution to \eqref{qde} with $d=1$. According
to \cite{BLS94}, such a solution exists at least in a small time interval if
the initial datum is smooth and positive.
Furthermore, let $q\in\N$, $q\ge 2$ (and typically $q\le 6$), be the order
of the backward differentiation formula. 

First, we consider the discrete variational derivative \eqref{q1.dvd}.
A Taylor expansion around $(t_{k+1},x_i)$ yields
\begin{align*}
  \frac{\delta F_d}{\delta(n(t_{k+1}), n(t_k))}\Big|_{x=x_i} 
  &= -\frac{\delta_i^{\langle2\rangle}(v(t_{k+1}, x_i) 
  + v(t_k, x_i))}{v(t_{k+1},x_i) + v(t_k, x_i)} 
  = \frac{v_{xx}}{v}(t_{k+1},x_i) + O(\tau) + O(h^2) \\
  &= \frac{\delta F}{\delta n}[n](t_{k+1},x_i) + O(\tau) + O(h^2),
\end{align*}
where $i=0,\ldots,N-1$, $k\ge 0$. Similarly,
$$
  \delta_i^{\langle1\rangle}\left(n(t_{k+1})\delta_i^{\langle1\rangle}
  \left(\frac{\delta F_d}{\delta(n(t_{k+1}),n(t_k))}\right)\right)\bigg|_{x=x_i} 
  = \left(n\left(\frac{\delta F[n]}{\delta n}\right)_x\right)_x(t_{k+1},x_i) 
  + O(\tau) + O(h^2).
$$
Thus, the local truncation error of the right-hand side in \eqref{q1.dvd} 
is of order $O(\tau)+O(h^2)$.
Since the left-hand side is of order $O(\tau)$ in time and exact at spatial
grid points $x_i$, the local truncation error  of scheme \eqref{q1.dvd} is
of order $O(\tau)+O(h^2)$. The monotonicity of the discrete Fisher information
is shown in \eqref{dvd.diss}. 

The mass conservation is an obvious consequence of the scheme. To prove the uniform
boundedness, we observe that, by the discrete $H^1$-seminorm,
$$
  \sum_{i=0}^{N-1}(\delta_i^+ V^k)^2 h \le F_d[U^0] < \infty\quad\text{for all 
  }k\geq1. $$
Then, according to the discrete Poincar\'e-Wirtinger inequality, for 
$i=0,\ldots,N-1$, $k\ge 1$,
\cite[Lemma 3.3]{FuMa10},
$$
  |V_i^k - M_k|^2 \le \sum_{i=0}^{N-1}(\delta_i^+ V^k)^2 h\le F_d[U^0]
$$
with $M_k = \sum_{i=0}^{N-1}V_i^kh$.  Jensen's inequality for the quadratic function and the mass conservation property of the method give $M_k\leq 1$ for all $k\ge 0$.
Finally, by the triangle inequality, $|V_i^k|\le F_d[U^0]^{1/2}+1$ and thus,
$|U_i^k|\le 2F_d[U^0]+2$. 

Next, we consider scheme \eqref{bdfq.dvdm} with the discrete variational
derivative \eqref{dvd.bdfq}. By construction, the left-hand side of 
\eqref{dvd.bdfq} is of order $q$ in time and exact at the spatial grid points
$x_i$. Thus, it remains to prove that the right-hand side is of order $(q,2)$
with respect to time-space discretization. 

Taylor expansions show, with a slight abuse of notation, that
\begin{align}
  \delta_i^\pm v(t_{k+1},x_i) 
  &= v_x(t_{k+1},x_i) \pm \frac{h}{2}v_{xx}(t_{k+1},x_i) + O(h^2), 
  \label{bdfq.aux.12} \\
  -\frac{\delta_i^{\langle2\rangle} v(t_{k+1},x_i)}{v(t_{k+1},x_i)} 
  &= -\frac{v_{xx}}{v}(t_{k+1},x_i) + O(h^2), \label{bdfq.aux.3} \\
  \frac{\delta_{k+1}^{1,q}\delta_i^{\langle2\rangle}n(t_{k+1},x_i)}{v(t_{k+1},x_i)}
  &= \frac{n_{txx}}{v}(t_{k+1},x_i) + O(\tau^q) + O(h^2), \label{bdfq.aux.4} \\
  \delta_i^\pm\delta_{k+1}^{1,q}n(t_{k+1},x_i) 
  &= n_{tx}(t_{k+1},x_i) \pm \frac{h}{2}n_{txx}(t_{k+1},x_i) + O(\tau^q) + O(h^2),
  \label{bdfq.aux.56} \\ 
  \delta_i^\pm\left(\frac{\delta_{k+1}^{1,q}n(t_{k+1},x_i)}{v(t_{k+1},x_i)}\right) 
  &= 2v_{tx}(t_{k+1},x_i) \pm hv_{txx}(t_{k+1},x_i) + O(\tau^q) + O(h^2).
  \label{bdfq.aux.78}
\end{align}
We prove that $r_{\rm corr}$ is of order $(q,2)$. Let $r_n$ and
$r_d$ denote the numerator and denominator of $r_{\rm corr}$, respectively,
replacing $V^{k+1}_i$ by $v(t_{k+1},x_i)$ and $U^{k+1}_i$ by $n(t_{k+1},x_i)$.
Taking into account the periodic boundary conditions, we find that
\begin{align*}
  \sum_{i=0}^{N-1} & (\delta_i^+\delta_{k+1}^{1,q}n(t_{k+1},x_i))\delta_i^+
  \left(\frac{\delta_{k+1}^{1,q}n(t_{k+1},x_i)}{v(t_{k+1},x_i)}\right)h \\
  &= \sum_{i=0}^{N-1}(\delta_i^-\delta_{k+1}^{1,q}n(t_{k+1},x_i))\delta_i^-
  \left(\frac{\delta_{k+1}^{1,q}n(t_{k+1},x_i)}{v(t_{k+1},x_i)}\right)h.
\end{align*}
Therefore, we can split $r_d$ into two parts:
\begin{align*}
  r_d &= \frac12\sum_{i=0}^{N-1}\Bigg[(\delta_i^+\delta_{k+1}^{1,q}n(t_{k+1},x_i))
  \delta_i^+\left(\frac{\delta_{k+1}^{1,q}n(t_{k+1},x_i)}{v(t_{k+1},x_i)}\right) \\ 
  &\phantom{xx}{}+ (\delta_i^-\delta_{k+1}^{1,q}n(t_{k+1},x_i))
  \delta_i^-\left(\frac{\delta_{k+1}^{1,q}n(t_{k+1},x_i)}{v(t_{k+1},x_i)}\right)
  \Bigg]h.
\end{align*}
In view of \eqref{bdfq.aux.56}-\eqref{bdfq.aux.78}, it follows that
$$
  r_d = 2\sum_{i=0}^{N-1}(n_{tx}v_{tx})(t_{k+1},x_i)h + O(\tau^q) + O(h^2).
$$
The numerator $r_n$ is treated in a similar way. Using \eqref{bdfq.aux.12},
the first term in $r_n$ can be written as 
\begin{align*}
  \delta_{k+1}^{1,q} F_d[n(t_{k+1})] 
  &= \frac12\frac{\dd}{\dd t}\sum_{i=0}^{N-1}\big((\delta_i^+v(t,x_i))^2 
  + (\delta_i^-v(t,x_i))^2\big)h\Big|_{t=t_{k+1}} + O(\tau^q) \\
  &= \sum_{i=0}^{N-1}(v_xv_{xt})(t_{k+1},x_i)h 
  + O(\tau^q) + O(h^2). 
\end{align*}
For the second term in $r_n$, we observe that, because of the periodic 
boundary conditions,
$$
  \sum_{i=0}^{N-1}\delta_i^+v(t_{k+1},x_i)
  \delta_i^+\left(\frac{\delta_{k+1}^{1,q}n(t_{k+1},x_i)}{v(t_{k+1},x_i)}\right)h
  =\sum_{i=0}^{N-1}\delta_i^-v(t_{k+1},x_i)
  \delta_i^-\left(\frac{\delta_{k+1}^{1,q}n(t_{k+1},x_i)}{v(t_{k+1},x_i)}\right)h,
$$
and hence, employing \eqref{bdfq.aux.12} and \eqref{bdfq.aux.78},
\begin{align*}
  \sum_{i=0}^{N-1}\delta_i^+v(t_{k+1},x_i)
  \delta_i^+\left(\frac{\delta_{k+1}^{1,q}n(t_{k+1},x_i)}{v(t_{k+1},x_i)}\right)h
  &= \frac12\sum_{i=0}^{N-1}\Bigg[\delta_i^+v(t_{k+1},x_i)
  \delta_i^+\left(\frac{\delta_{k+1}^{1,q}n(t_{k+1},x_i)}{v(t_{k+1},x_i)}\right) \\
  &\phantom{xx}{}+ \delta_i^-v(t_{k+1},x_i)\delta_i^-
  \left(\frac{\delta_{k+1}^{1,q}n(t_{k+1},x_i)}{v(t_{k+1},x_i)}\right)\Bigg]h \\
  &= 2\sum_{i=0}^{N-1}(v_xv_{xt})(t_{k+1},x_i)h + O(\tau^q) + O(h^2).
\end{align*}
Summarizing these identities yields $r_{\rm corr}=O(\tau^q) + O(h^2)$. 
Finally, \eqref{bdfq.aux.3}-\eqref{bdfq.aux.4} imply that
\begin{align*}
  \frac{\delta F_d}{\delta(n(t_{k+1}),\ldots,n(t_{k+1-q}))}\Big|_{x=x_i} 
  &= -\frac{\delta_i^{\langle2\rangle} v(t_{k+1},x_i)}{v(t_{k+1},x_i)} 
  - r_{\rm corr}
  \frac{\delta_{k+1}^{1,q}\delta_i^{\langle2\rangle}n(t_{k+1},x_i)}{v(t_{k+1},x_i)} \\
  &= \frac{\delta F[n]}{\delta n}(t_{k+1}, x_i) + O(\tau^q) + O(h^2).
\end{align*}
This shows that the discrete variational derivative \eqref{dvd.bdfq} is of
order $q$ in time, finishing the proof.
\end{proof}


\section{Numerical examples}\label{sec.num}

In this section, we present some numerical examples which illustrate
the decay properties of the entropy functionals and Fisher information
as well as the convergence properties of the schemes presented in the previous sections.

\subsection{BDF2 finite-difference scheme}

The DLSS equation \eqref{dlss} is approximated by the BDF2 method in time
and central finite differences in space. The scheme is given by the following
nonlinear system with unknowns $V_i^k=(U_i^k)^{\alpha/2}$: For
$i=0,\ldots,N-1$ and $k=1$,
$$
  (V_i^{1})^{2/\alpha-1}\big(V_i^{1} - V_i^0 \big)
  + \tau\delta_i^{\langle 2\rangle}\left((V_i^{1})^{2/\alpha}
  \delta_i^{\langle 2\rangle}\log V_i^{1}\right) = 0
$$
and for $i=0,\ldots,N-1$, $k\ge 2$,
$$
  (V_i^{k+1})^{2/\alpha-1} \left(\frac32 V_i^{k+1} - 2V_i^k + \frac12 V_i^{k-1}
  \right) + \tau\delta_i^{\langle 2\rangle}\left((V_i^{k+1})^{2/\alpha}
  \delta_i^{\langle 2\rangle}\log V_i^{k+1}\right) = 0. 
$$
The initial datum $(V_i^0)$ is given by $(n_0(x_i)^{\alpha/2})$. For $k=1$,
the scheme corresponds to the implicit Euler discretization, needed to
initialize the BDF2 scheme for $k\ge 2$. The above nonlinear system, with periodic
boundary conditions, is solved using the Newton method.

We choose the initial datum $n_0(x)=0.001+\cos^{16}(\pi x)$, $x\in[0,1]$.
The spatial mesh size is $h=0.005$ ($N=200$) and the time step $\tau=10^{-6}$.
The (continuous) entropies $E_\alpha[n]$ are dissipated for $1\le\alpha<3/2$. Figure
\ref{fig.bdf2.stab} (a) illustrates the stability and, in fact, decay of the
discrete entropies $E_{\alpha,\dd}$, defined below, 
for various values of $\alpha$. Although
Theorem~\ref{thm.bdf2.ex} does not provide a stability estimate for $\alpha=1$,
the numerical results indicate that the discrete entropy
$E_{1,\dd}[U]=\sum_{i=0}^{N-1}(U_i(\log U_i-1)+1)h$ is decreasing. 
Figure \ref{fig.bdf2.stab} (b)
shows that the decay of the discrete relative entropy is exponential, 
and even the discrete Fisher information converges exponentially fast to zero.
Here, the discrete relative entropy is defined by 
$$
  E_{\alpha,\dd}^{rel}[U^k] = E_{\alpha,\dd}[U^k] - E_{\alpha,\dd}[\bar U], \quad
  \mbox{where }
  E_{\alpha,\dd}[U^k] = \sum_{i=0}^{N-1}(U_i^k)^\alpha h, \
  \bar U = \sum_{i=0}^{N-1}U^k_i h.
$$

\begin{figure}[ht]
\subfloat[\label{fig:entropy_decay}]{
\includegraphics[width=75mm]{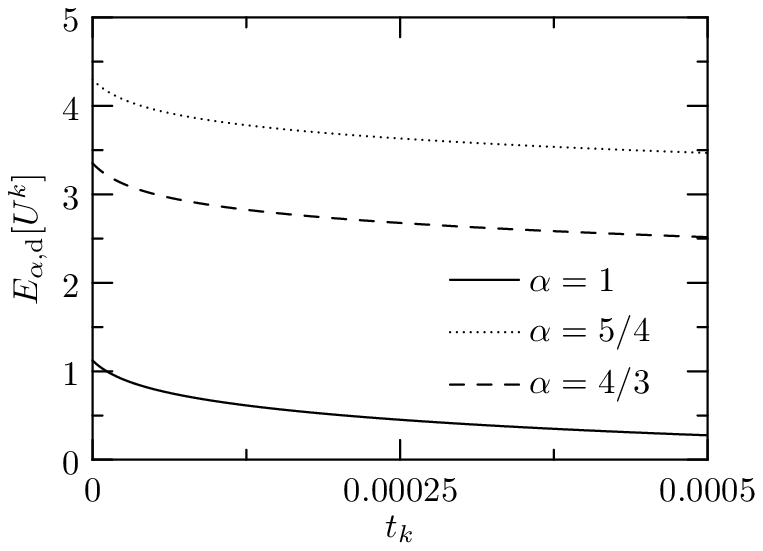}}
\hfill 
\subfloat[\label{fig:rel_entropy_decay}]{
\includegraphics[width=75mm]{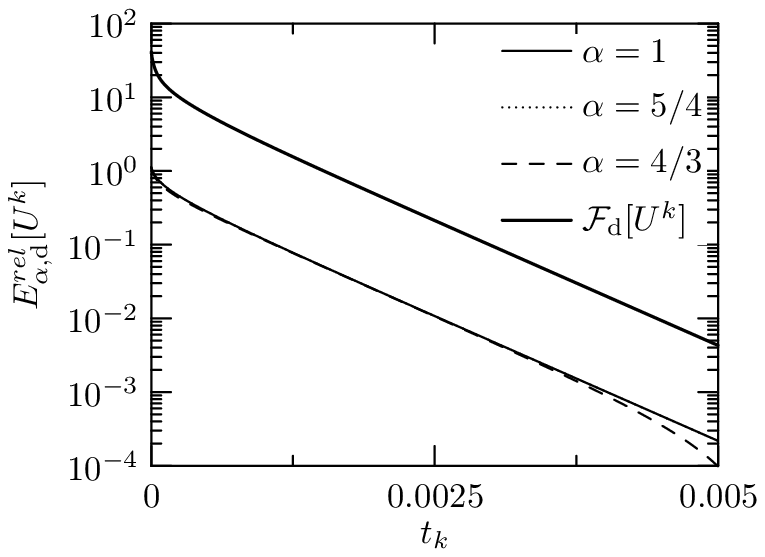}}
\caption{
(a) Entropy stability (decay) for the BDF2 finite-difference scheme.
(b) Exponential decay of the discrete relative entropy and the discrete Fisher
information for the BDF2 finite-difference scheme.}
\label{fig.bdf2.stab}
\end{figure}

According to Theorem \ref{thm.bdf2.conv}, the semi-discrete BDF2 scheme 
converges in second order if $\alpha=1$. This may be not the case for the
fully discrete scheme, since the discretization may destroy the monotonicity
structure of the spatial operator. However, Figure \ref{fig.bdf2.conv} shows
that the numerical convergence rate is close to 2, even for $\alpha\neq 1$. 
The numerical convergence rates $cr$ have been obtained by the linear regression method. 
The convergence of the method is measured in the discrete $\ell^2$-norm
$$
  \|e_m\|_2 := \left(\sum_{i=0}^{N-1}(V_{{\rm ex},i}^m-V_i^m)^2 h\right)^{1/2},
$$
and the numerical solutions are compared at time $t=5\cdot 10^{-5}$.
Here, the ``exact'' solution $V_{{\rm ex},i}^m$ is computed by the above scheme
using the very small time step $\tau=10^{-10}$. 

\begin{figure}[ht]
\centering\includegraphics[width=75mm]{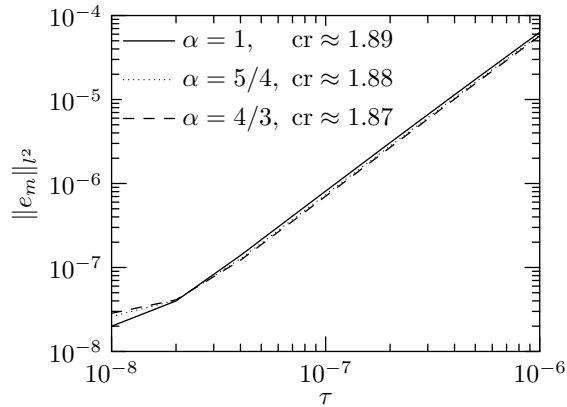}
\caption{Temporal convergence of the BDF2 finite-difference
scheme for various values of $\alpha$; the convergence rate is denoted by $cr$.}
\label{fig.bdf2.conv}
\end{figure}


\subsection{Discrete variational derivative method}

We present some numerical results obtained from the DVD and BDF$q$ DVD schemes derived
in Section \ref{sec.dvdm}. The initial datum and the numerical parameters are chosen
as in previous subsection. In order to solve the discrete nonlinear systems,
we employed here the NAG toolbox routine {\tt c05nb},
which is based on a modification of the Powell hybrid method. 
It turned out that this routine is at least three times faster than the standard
MATLAB routine {\tt fsolve}. 

In Figure \ref{fig.dvdm.fish}, the temporal evolution
of the discrete relative entropies $E_\alpha^{\rm rel}[U^k]$ and the 
discrete Fisher information $F_d[U^k]$ are depicted for (a) the implicit Euler scheme 
\eqref{bdf1.dvdm} and (b) the BDF2 scheme \eqref{bdfq.dvdm}. 
We observe that the decay is in all cases exponential. 
This holds also true for the BDF3 scheme (results not shown).

\begin{figure}[ht]
\subfloat[\label{fig:dvdm_fisher_decay}]{
\includegraphics[width=75mm]{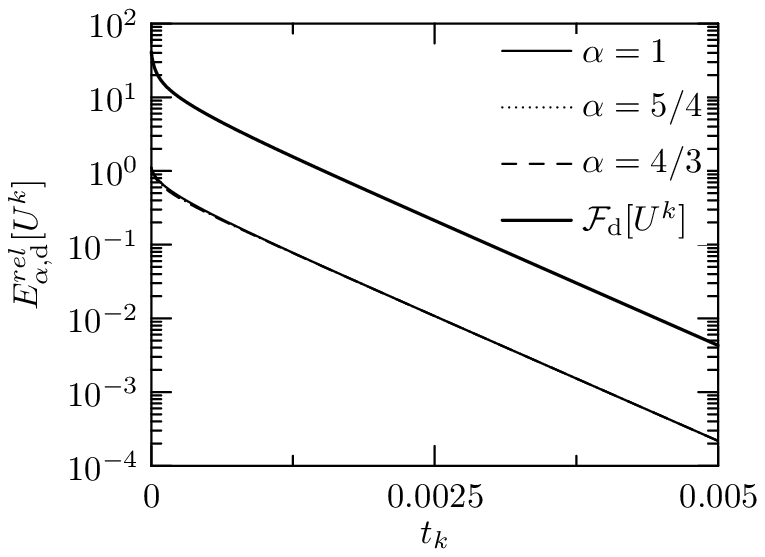}}
\hfill
\subfloat[\label{fig:bdf2_dvdm_fisher_decay}]{
\includegraphics[width=75mm]{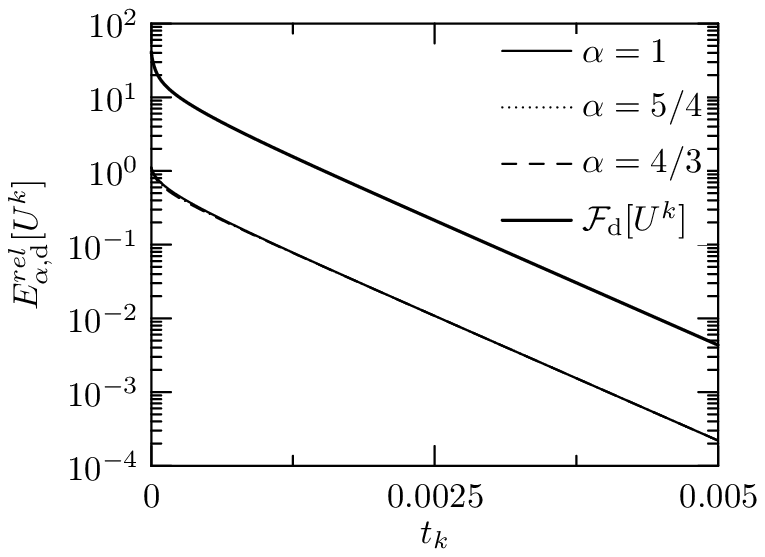}}
\caption{Exponential decay of the discrete Fisher
information and relative entropies using (a) the DVD scheme
and (b) the BDF2 DVD scheme.}
\label{fig.dvdm.fish}
\end{figure}

Next, we test numerically the convergence in time of the DVD scheme.
Figure \ref{fig.dvdm.conv} illustrates the $\ell^2$-errors of the methods.
We have chosen the mesh size $h=0.01$, and we compared the numerical solutions
at time $t_m=5\cdot 10^{-5}$. The ``exact'' solutions are computed by the 
respective method taking the time step $\tau=10^{-9}$. The numerical
convergence rates, computed by the linear regression method, are
given in Table \ref{tab.conv}. We note that the BDF3 DVD scheme gives only
slightly better results than the BDF2 DVD scheme. The reason is that
the first step is initialized by the first-order scheme \eqref{bdf1.dvdm},
and this initialization error cannot be compensated by the higher-order
accuracy of the local approximation. In order to obtain a third-order scheme,
we need to initialize the scheme with a second-order discretization.

\begin{figure}[ht]
\centering\includegraphics[width=75mm]{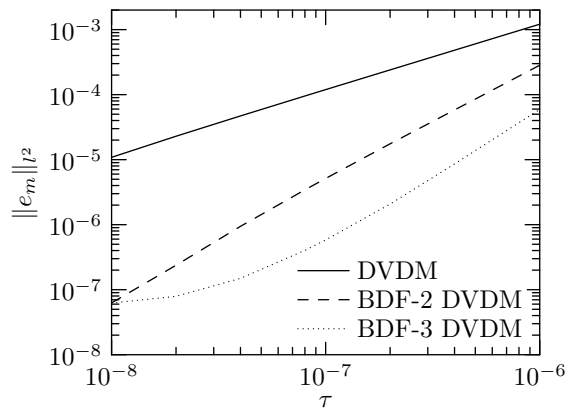}
\caption{Temporal convergence of the DVD, BDF2 DVD, and BDF3 DVD schemes.}
\label{fig.dvdm.conv}
\end{figure}

\begin{table}[ht]
\centering\begin{tabular}{cc}
\hline
Scheme & Convergence rate\\
\hline
DVD       & $1.020$\\
BDF2 DVD & $1.824$\\
BDF3 DVD & $1.977$ \\
\hline
\end{tabular}
\vspace{1mm}
\caption{Numerical temporal convergence rates for the discrete variational
derivative methods.}
\label{tab.conv}
\end{table}



\begin{thebibliography}{11}
\bibitem{BBG98} J.~Barrett, J.~Blowey, and H.~Garcke. Finite element approximation of 
a fourth order nonlinear degenerate parabolic equation. 
{\em Numer. Math.} 80 (1998), 525--556.

\bibitem{BLS94} P.~Bleher, J.~Lebowitz, and E.~Speer. Existence and 
positivity of solutions of a fourth-order nonlinear PDE describing interface
fluctuations. {\em Commun. Pure Appl. Math.} 47 (1994), 923--942.

\bibitem{BJM11} M.~Bukal, A.~J\"ungel, and D.~Matthes. A multidimensional nonlinear
sixth-order quantum diffusion equation. Submitted for publication, 2011.

\bibitem{CJT03} J.~A.~Carrillo, A.~J\"ungel, and S.~Tang. Positive entropic
schemes for a nonlinear fourth-order equation. 
{\em Discrete Contin. Dynam. Sys. B} 3 (2003), 1--20.

\bibitem{CGJ11} C.~Chainais-Hillairet, M.~Gisclon, and A.~J\"ungel. 
A finite-volume scheme for the multidimensional quantum drift-diffusion model 
for semiconductors. {\em Numer. Meth. Part. Differ. Eqs.} 27 (2011), 1483--1510.

\bibitem{Dah78} G.~Dahlquist. $G$-stability is equivalent to $A$-stability.
{\em BIT} 18 (1978), 384--401.

\bibitem{DMR05} P.~Degond, F.~M\'ehats, and C.~Ringhofer. Quantum energy-transport 
and drift-diffusion models. {\em J. Stat. Phys.} 118 (2005), 625--665.

\bibitem{DLSS91} B.~Derrida, J.~Lebowitz, E.~Speer, and H.~Spohn. 
Fluctuations of a stationary nonequilibrium interface.
{\em Phys. Rev. Lett.} 67 (1991), 165--168.

\bibitem{DMM10} B.~D\"uring, D.~Matthes, and J.-P.~Mili\v{s}i\'c. A gradient 
flow scheme for nonlinear fourth order equations. 
{\em Discrete Contin. Dyn. Syst. B} 14 (2010), 935--959.

\bibitem{Emm04} E.~Emmrich. Error of the two-step BDF for the incompressible 
Navier-Stokes problem. {\em M2AN Math. Model. Numer. Anal.} 38 (2004), 757--764.

\bibitem{Emm05} E.~Emmrich. Stability and error of the variable two-step BDF for 
semilinear parabolic problems. {\em J. Appl. Math. Comput.} 19 (2005), 33--55.

\bibitem{Emm09} E.~Emmrich. Two-step BDF time discretization of nonlinear
evolution problems governed by monotone operators with strongly continuous
perturbations. {\em Comput. Meth. Appl. Math.} 9 (2009), 37--62.

\bibitem{FuMa10} D.~Furihata and T.~Matsuo. {\em Discrete Variational Derivative
Method}. Chapman and Hall/CRC Press, Boca Raton, Florida, 2010.

\bibitem{GST09} U.~Gianazza, G.~Savar\'e, and G.~Toscani. 
The Wasserstein gradient flow of the Fisher information and the quantum 
drift-diffusion equation. {\em Arch. Ration. Mech. Anal.} 194 (2009), 133--220.

\bibitem{GiRa81} V.~Girault and P.-A.~Raviart. {\em  Finite Element Approximation 
of the Navier-Stokes Equations}. Lect. Notes Math. 749. Springer, Berlin, 1981.

\bibitem{GlGa09} A.~Glitzky and K.~G\"artner. Energy estimates of continuous and
discretized electro-reaction-diffusion systems. {\em Nonlin. Anal., Theory Methods Appl.} 70 (2009),
788--805.

\bibitem{Gru03} G.~Gr\"un. On the convergence of entropy consistent schemes for 
lubrication type equations in multiple space dimensions.
{\em Math. Comput.} 72 (2003), 1251--1279.  

\bibitem{GrRu00} G.~Gr\"un and M.~Rumpf. Nonnegativity preserving convergent schemes 
for the thin film equation. {\em Numer. Math.} 87 (2000), 113--152.

\bibitem{HiSu00} A.~Hill and E.~S\"uli. Approximation of the global attractor for 
the incompressible Navier-Stokes equations. {\em IMA J. Numer. Anal.} 20 (2000), 
633--667.

\bibitem{Kre78} H.~Kreth. Time-discretisations for nonlinear evolution equations. 
In: {\em Numerical Treatment of Differential Equations in Applications} 
(Proc. Meeting, Math. Res. Center, Oberwolfach, 1977), pp. 57--63. Lect. Notes Math.
679. Springer, Berlin, 1978.

\bibitem{Jue09} A.~J\"ungel. {\em Transport Equations for Semiconductors}. 
Lect.~Notes Phys.~773, Springer, Berlin, 2009.

\bibitem{JuMa08} A.~J\"ungel and D.~Matthes. The Derrida-Lebowitz-Speer-Spohn 
equation: existence, non-unique\-ness, and decay rates of the solutions.
{\em SIAM J. Math. Anal.} 39 (2008), 1996--2015.

\bibitem{JuMi09} A.~J\"ungel and J.-P.~Mili\v{s}i\'{c}. A sixth-order nonlinear 
parabolic equation for quantum systems. {\em SIAM J. Math. Anal.} 41 (2009), 
1472--1490.

\bibitem{JuPi00} A.~J\"ungel and R.~Pinnau. Global non-negative solutions of a
nonlinear fourth-oder parabolic equation for quantum systems.
{\em SIAM J. Math. Anal.} 32 (2000), 760--777.

\bibitem{JuPi01} A.~J\"ungel and R.~Pinnau. A positivity preserving numerical
scheme for a nonlinear fourth-order parabolic equation.
{\em SIAM J. Numer. Anal.} 39 (2001), 385--406.

\bibitem{JuPi03} A.~J\"ungel and R.~Pinnau. 
Convergent semidiscretization of a nonlinear fourth order parabolic
system. {\em M2AN Math.~Model.~Numer.~Anal.} 37 (2003), 277--289.

\bibitem{JuVi07} A.~J\"ungel and I.~Violet. First-order entropies for the
Derrida-Lebowitz-Speer-Spohn equation. {\em Discrete Contin. Dyn. Syst. B} 
8 (2007), 861--877.

\bibitem{Moo94} P.~Moore. A posteriori error estimation with finite element 
semi- and fully discrete methods for nonlinear parabolic equations in one space 
dimension. {\em SIAM J. Numer. Anal.} 31 (1994), 149--169.

\bibitem{WW10} M.~Westdickenberg and J.~Wilkening. Variational particle schemes
for the porous medium equation and for the system of isentropic Euler equations.
{\em M2AN Math.~Model.~Numer.~Anal.} 44 (2010), 133--166.

\bibitem{WiWo65} D.~Willett and J.~Wong. On the discrete analogues of some
generalizations of Gronwall's inequality. {\em Monatsh. Math.} 69 (1965), 362--367.

\bibitem{ZhBe00} L.~Zhornitskaya and A.~Bertozzi. Positivity-preserving numerical
schemes for lubrication-type equations. {\em SIAM J. Numer. Anal.} 37 (2000), 523--555.

\end{thebibliography}
\end{document}